\documentclass[11pt,reqno]{article}
\usepackage{booktabs} % for much better looking tables
\usepackage{array} % for better arrays (eg matrices) in maths
\usepackage{paralist} % very flexible & customisable lists (eg. enumerate/itemize, etc.)

\usepackage{verbatim} % adds environment for commenting out blocks of text & for better verbatim
\usepackage{amsthm}
\usepackage[table,xcdraw]{xcolor}

\usepackage{latexsym, amsmath, amssymb, mathtools, a4, epsfig, color}
\usepackage{blindtext}
\usepackage{graphicx}
\usepackage[utf8]{inputenc}
\usepackage[export]{adjustbox}
\usepackage{wrapfig}
\usepackage{apptools}

\usepackage{floatrow}

\usepackage[shortlabels]{enumitem}
\usepackage[margin=0pt]{subfig}
\usepackage[normalem]{ulem} % to use strikethrough %\sout{text}

\AtAppendix{\counterwithin{definition}{subsection}}
\AtAppendix{\counterwithin{theorem}{subsection}}

\graphicspath{}

\newtheorem{theorem}{Theorem}[section]

\newtheorem{lemma}[theorem]{Lemma}
\newtheorem{prop}[theorem]{Proposition}
\newtheorem{definition}{Definition}

\newcommand{\bfc}{\mathbf{c}}

\newcommand{\bfp}{\mathbf{p}}
\newcommand{\bfq}{\mathbf{q}}

\newcommand{\bfx}{\mathbf{x}}
\newcommand{\bfy}{\mathbf{y}}
\newcommand{\bfz}{\mathbf{z}}

\newcommand{\RR}{\mathbb{R}}
\newcommand{\CC}{\mathbb{C}}

\newcommand{\Om}{\Omega}
\newcommand{\la}{\langle}
\newcommand{\ra}{\rangle}

\newcommand{\SingleOmega}{\mathcal{S}_{\partial\Omega}}

\newcommand{\KstarOmega}{\mathcal{K}_{\partial\Omega}^{*}}

\newcommand{\Kcal}{\mathcal{K}}
\newcommand{\Scal}{\mathcal{S}}

\newcommand{\SingleD}{\mathcal{S}_{\partial D}}
\newcommand{\KstarD}{\mathcal{K}_{\partial D}^{*}}
\newcommand{\KstarOm}{\mathcal{K}_{\partial \Omega}^{*}}

\newcommand{\KD}{\mathcal{K}_{\partial D}}
\newcommand{\Rtwo}{\mathbb{R}^2}
\newcommand{\LtwobdOmega}{L^2(\partial\Omega)}
\newcommand{\LtwozerobdOmega}{L_0^2(\partial\Omega)}

\newcommand{\Hstar}{{\mathcal{H}}^*}

\newcommand{\SingleBR}{{\mathcal{S}_{\partial B_R}}}
\newcommand{\SingleBr}{{\mathcal{S}_{\partial B_r}}}
\newcommand{\SingleBmr}{{\mathcal{S}_{\partial B_{-r}}}}

\newcommand{\KstarBR}{{\mathcal{K}_{\partial B_R}^{*}}}
\newcommand{\KstarBr}{{\mathcal{K}_{\partial B_r}^{*}}}
\newcommand{\KstarBmr}{{\mathcal{K}_{\partial B_{-r}}^{*}}}

\newcommand{\chipBR}{\chi_{\partial B_R}}

\newcommand{\tvarphi}{{\widetilde{\varphi}}}
\newcommand{\tpsi}{{\widetilde{\psi}}}
\newcommand{\hvarphi}{{\widehat{\varphi}}}
\newcommand{\hpsi}{{\widehat{\psi}}}
\newcommand{\tq}{{\tilde{q}}}

\newcommand{\eqnref}[1]{(\ref {#1})}
\newcommand{\p}{\partial}
\newcommand{\beq}{\begin{equation}}
\newcommand{\eeq}{\end{equation}}

\newcommand{\LtwoD}{L^2(\partial D)}

\newcommand{\LtwoOm}{L^2(\partial \Omega)}
\numberwithin{equation}{section}
\numberwithin{figure}{section}

\begin{document}

\title{Spectral analysis of the Neumann--Poincar\'{e} operator on the crescent-shaped domain and touching disks and analysis of plasmon resonance
%\title{Series expansion (Shape basis) for interface problem for general shape domains
 \thanks{\footnotesize
This work is supported by the Basic Science Research Program through the National Research Foundation of Korea(NRF) grants NRF-2019R1A6A1A10073887, NRF-2016R1A2B4014530 and NRF-2021R1A2C1011804.}
}
\author{
Younghoon Jung\thanks{\footnotesize Department of Mathematical Sciences, Korea Advanced Institute of Science and Technology, Daejeon 34141, Korea (hapy1010@kaist.ac.kr).}  \and
Mikyoung Lim\thanks{\footnotesize Department of Mathematical Sciences, Korea Advanced Institute of Science and Technology, Daejeon
34141, Korea (mklim@kaist.ac.kr).} 
}

\date{\today}
\maketitle
\begin{abstract}
We consider the Neumann--Poincar\'{e} operator on a planar domain enclosed by two touching circular boundaries. This domain, which is a crescent-shaped domain or touching disks, has a cusp at the touching point of two circles. We analyze the operator via the Fourier transform on the boundary circles of the domain. In particular, we define a Hilbert space on which the operator is bounded, self-adjoint. We then obtain the complete spectral resolution of the Neumann--Poincar\'{e} operator. On both the crescent-shaped domain and touching disks, the Neumann--Poincar\'{e} operator has only absolutely continuous spectrum on the closed interval $[-1/2,1/2]$. As an application, we analyze the plasmon resonance on the crescent-shaped domain and touching disks.

%\vskip .5cm
%
%\begin{center}
%\textbf{R\'{e}sum\'{e}}
%\end{center}
%
%Nous examinons l'op\'{e}rateur Neumann--Poincar\'{e} sur un domaine planaire entour\'{e} de deux fronti\`{e}res circulaires qui se touchent. Ce domaine, qui est un domaine en forme de croissant ou des disques en contact, a un point de rebroussement au point de contact de deux cercles. 
%Nous analysons cet op\'{e}rateur via la transform\'{e}e de Fourier sur les cercles des fronti\`{e}res du domaine. En particulier, nous d\'{e}finissons un espace de Hilbert sur lequel l'op\'{e}rateur Neumann--Poincar\'{e} est born\'{e}, auto-adjoint. Nous obtenons alors la r\'{e}solution spectrale compl\`{e}te de l'op\'{e}rateur Neumann--Poincar\'{e}. Sur un domaine en forme de croissant et aussi sur des disques en contact, l'op\'{e}rateur Neumann--Poincar\'{e} n'a que le spectre absolument continu sur l'intervalle ferm\'{e} $[-1/2,1/2]$. En guise d'application, nous analysons la r\'{e}sonance plasmonique sur un domaine en forme de croissant et des disques en contact.
\end{abstract}

 

\tableofcontents

\section{Introduction}
Spectral analysis of the Neumann--Poincar\'{e} (NP) operator has received much attention in recent years due to its applications to electromagnetic problems in metamaterials, such as localized surface plasmon resonance of nanoparticles and invisibility cloaking \cite{Ammari:2013:STN,Ando:2016:PRF,Bonnetier:2012:PBO,Bonnetier:2013:SPV,Grieser:2014:PEP,Helsing:2017:CSN,Mayergoyz:2005:EPR,Milton:2006:CEA}. The NP operator is a singular integral operator which naturally appears when one solves interface problems for the Laplacian by using the layer potentials. More precisely, given a simply connected bounded Lipschitz domain $\Om\subset\RR^2$, the NP operator $\Kcal_{\p \Om}^*$ is defined by
\beq\label{kstar:def}
  \KstarOm[\psi](x)=p.v.\, \frac{1}{2\pi}\int_{\partial D}\frac{\langle x-y,\nu_x\rangle}{|x-y|^2}\, \psi(y)\, d\sigma(y),\quad x\in \partial \Om,
\eeq
for a density function $\psi\in\LtwoOm$, where $p.v$. denotes the Cauchy principal value, and $\nu_x$ denotes the outward unit normal vector to $\partial \Om$ at $ x\in\partial \Om$. The NP operator is similarly defined in three dimensions \cite{Kellogg:1953:FPT,Verchota:1984:LPR}.
In this paper, we investigate the spectrum of the Neumann--Poincar\'{e} operator on a planar domain that is enclosed by two touching circular boundaries: a crescent-shaped domain and touching disks (see Figure \ref{fig:exampledomains}). This domain has a cusp at the touching point of the two boundary circles.  To the best knowledge of the authors, this is the first article providing the spectral resolution of the NP operator on a planar domain with a cusp.
\begin{figure}[h!]
    \subfloat{\includegraphics[height=5.2cm, width=4.6cm, trim={3.6cm 2cm 2cm 2cm}, clip]{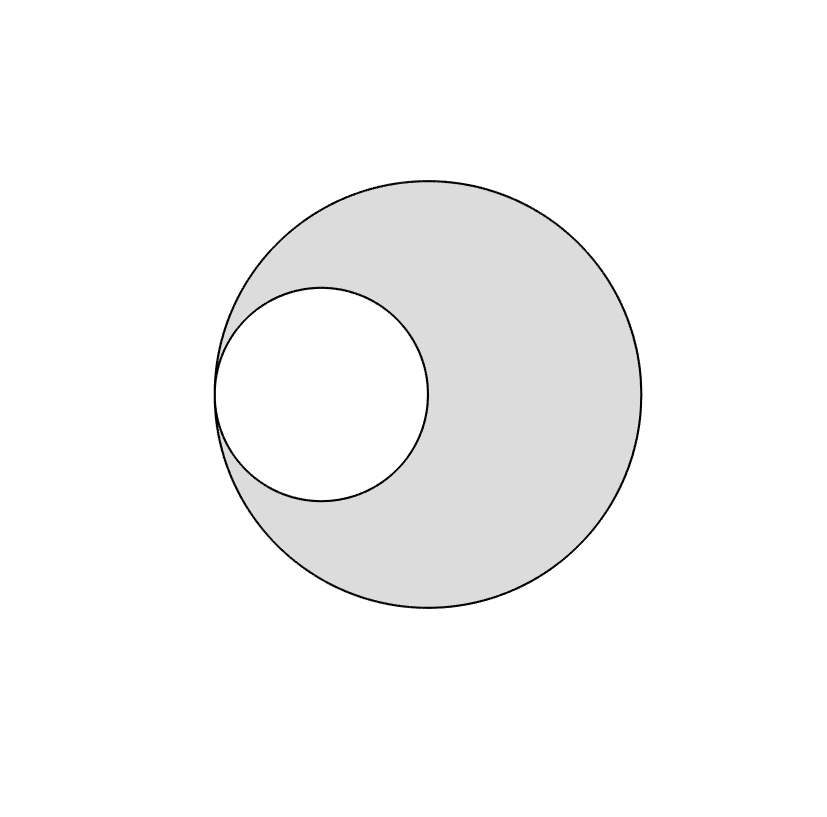}}
	\subfloat{\includegraphics[height=5.2cm,width=4.6cm, trim={3.6cm 2.2cm 2cm 2.2cm}, clip]{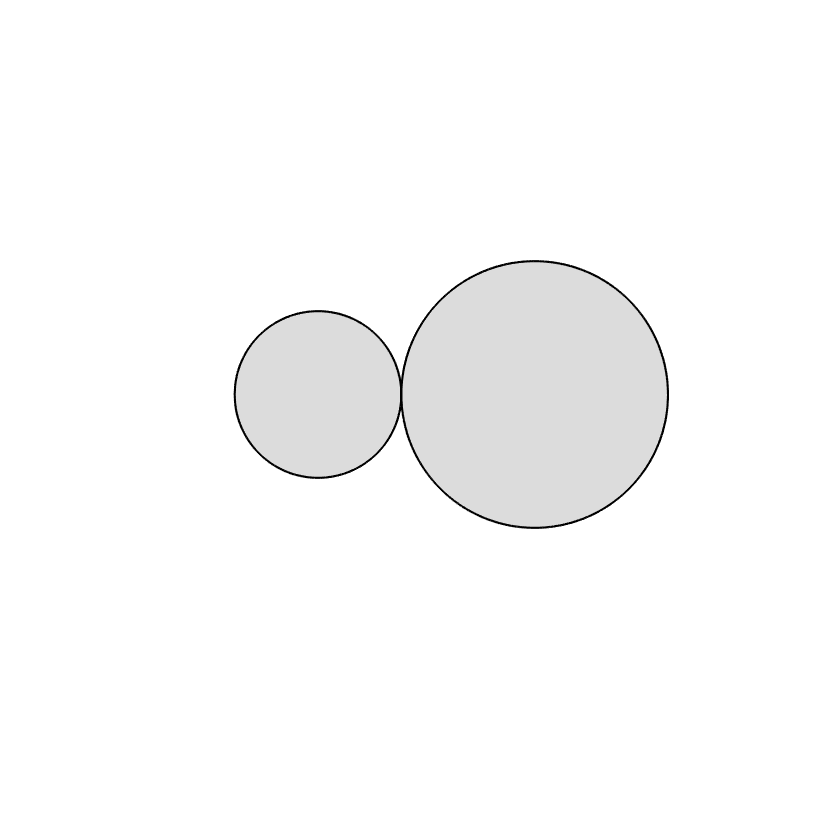}}
	\caption{A crescent-shaped domain (left) and touching disks (right).}
	\label{fig:exampledomains}
\end{figure}

For the case of a Lipschitz domain (without a cusp point on its boundary), the spectral property of the NP operator has been studied in many literatures. Let us review some essential results. We refer to a review article \cite{Ando:2020:SAN} and the references therein for more results. 
The NP operator is not symmetric on $L^2(\partial \Om)$ unless $\Om$ is a disk or a ball \cite{Lim:2001:SBI}. 
However, it can be realized as a self-adjoint operator on $H^{-1/2}_0(\partial \Om)$ with a new inner-product which is differently defined but equivalent to the original inner-product, based on Plemelj's symmetrization principle \cite{Ando:2016:APR,Kang:2016:SPO,Khavinson:2007:PVP}. 
Here, $H^{-1/2}_0(\partial \Om)$ is the Sobolev space $H^{-1/2}(\partial\Omega)$ with the mean-zero condition. We denote by $\mathcal{H}^*$ the space $H_0^{-1/2}(\partial\Omega)$ equipped with the new inner product. Since $\KstarOm$ is self-adjoint on $\mathcal{H}^*$, its spectrum $\sigma(\KstarOm)$ is a closed set contained in the real line. In fact, the spectrum of the NP operator on $L^2_0(\partial \Om)$ (or $\mathcal{H}^*$) is contained in $(-1/2,1/2)$  \cite{Escauriaza:2004:TPS,Fabes:1992:SRC,Kellogg:1953:FPT} (see also \cite{Kang:2018:SPS,Krein:1998:CLO} for the permanence of the spectrum for the NP operator with different norms).

For a $C^{1,\alpha}$ domain, $\KstarOm$ on $\mathcal{H}^{*}$ is compact as well as symmetric so that it admits only discrete real eigenvalues, namely $\lambda_j$, as its spectrum, where zero is the only possible accumulation point. 
The NP operator admits the spectral decomposition
\begin{equation}\label{eqn:spectraldecompintro}
  \KstarOm = \sum_{j=1}^\infty\lambda_j\psi_j\otimes\psi_j,
\end{equation}
where $\psi_j$ are eigenvectors corresponding to eigenvalues $\lambda_j$.
We refer to \cite{Ando:2018:EDE,Jung:2020:DEE,Miyanishi:2017:EED} for the decay estimates for the eigenvalues of the NP operator on a planar domain (see also \cite{Blumenfeld:1914:UPF} for the symmetricity of the spectrum on a planar domain). For simple shapes such as disks or ellipses, the complete sets of eigenvalues are known \cite{Ammari:2007:BIM}. 
We refer to  \cite{Ammari:2007:BIM, Feng:2016:SNP} for the eigenvalues of $\KstarOm$ of an ellipse or an ellipsoid and to \cite{Ando:2019:SSN} for the eigenvalue property of the NP operator of tori. The NP operator can also be defined for domains with separated components. For example, the eigenvalues of a domain consisting of two separated unit disks were explicitly expressed in terms of the distance between the disks \cite{Bonnetier:2012:PBO} (see also \cite{Bonnetier:2013:SPV,Lim:2015:AOT}).

For a Lipschitz domain with corners, the NP operator admits a continuous spectrum as well as eigenvalues, and it can be decomposed into three parts: the absolutely continuous spectrum, singularly continuous spectrum, and pure point spectrum (namely, $\sigma_{ac}(\KstarOm)$, $\sigma_{sc}(\KstarOm)$ and $\sigma_{pp}(\KstarOm)$, respectively).
It holds the following integral expression by the spectral theorem on a bounded, self-adjoint operator on a Hilbert space (see, for instance, \cite{Teschl:2009:MMQ,Yosida:1995:FA}):
\begin{equation}\label{eqn:generalspectraltheorem}
\KstarOm = \int_{\sigma(\KstarOm)}t \, d\mathbb{E}(t),	
\end{equation}
where $\mathbb{E}(t)$ is a resolution of identity for $\KstarOm$. 
Various studies have investigated the spectral properties of the NP operator for cornered domains \cite{Helsing:2017:CSN,Kang:2017:SRN,Li:2019:EEN,Perfekt:2014:SBN,Perfekt:2017:ESN}.
It was shown that the essential spectrum on a bounded planar domain with corners is an interval determined by the corner angles \cite{Helsing:2017:CSN,Kang:2017:SRN,Perfekt:2021:PEP,Perfekt:2014:SBN,Perfekt:2017:ESN}. 
For intersecting disks, the complete spectral resolution of the NP operator was obtained \cite{Kang:2017:SRN}, where $\sigma(\KstarOm)$ consists of only the absolutely continuous spectrum. 
In \cite{Helsing:2017:CSN}, a numerical method to determine the spectrum was developed by using its relation to the plasmon resonance rate. We refer to see \cite{Ando:2016:APR,Ando:2020:SAN,Dhia:2020:CMP,Bonnetier:2019:PRB,Bonnetier:2019:CES,Li:2019:EEN} for more results on the spectral properties of the NP operator and plasmon resonance.
Numerical examples obtained in \cite{Helsing:2017:CSN} show that all three kinds of spectrums (that is, $\sigma_{ac}(\KstarOm)$, $\sigma_{sc}(\KstarOm)$ and $\sigma_{pp}(\KstarOm)$) can appear depending on the domains. The existence of embedded eigenvalues within the essential spectrum was verified numerically \cite{Helsing:2017:CSN} and analytically \cite{Li:2019:EEN}. Also, it was shown that infinitely many embedded eigenvalues appear for a perturbed sphere by smoothly attaching a conical singularity \cite{Li:2020:IME}.

Touching disks can be considered as the limiting geometry of two different types of shapes. One is the limiting geometry of two separated disks as the distance between them tends to zero, and the other is that of intersecting disks whose external corner angle tends to zero.
The spectrum $\sigma(\KstarOm)$ for the domain consisting of two separated unit disks with the distance $\delta>0$ is a sequence of eigenvalues given by (see \cite{Bonnetier:2012:PBO})
\begin{equation}\label{eqn:twosepdiskseig}
\lambda_n^{\pm}=\pm\frac{(1+\delta-\sqrt{\delta(2+\delta)}\,)^{2n}}{2},\quad n=1,2,\dots.
\end{equation}
As $\delta$ tends to zero, $\lambda_n^{\pm}$ become densely located in $[-1/2,1/2]$.
For intersecting disks with external angle $\theta\in(0,\pi)$ at the corner points, it holds that (see \cite{Kang:2017:SRN})
\begin{equation}\label{eqn:intersectingdisksspect}
\sigma(\KstarOm)=\sigma_{ac}(\KstarOm)=[-b,b],\quad b = |1/2-\theta/\pi|.
\end{equation}
As the external angle $\theta$ tends to zero, the spectral bound $b$ tends to $1/2$. The extension of the results on the spectrum of the separated disks and intersecting disks,  \eqnref{eqn:twosepdiskseig} and \eqnref{eqn:intersectingdisksspect}, to touching disks is not straightforward as it involves the layer potential technique on a domain with a cusp point, which, to the best knowledge of the authors, has not been established. We need to generalize the NP operator defined on a Lipschitz domain to the considered domains with a cusp point in a suitable function space.

In the present paper, for $\Om$ a crescent-shaped domain and touching disks, we first define the NP operator on $L^2_0(\p\Om)$ as a boundary integral operator similarly to the Lipschitz domain case. We then further investigate the NP operator via the Fourier transform on the boundary circles of the domain. More precisely, two touching circular boundaries of the domain are mapped to two parallel lines with a M\"{o}bius transformation, and the Fourier transform is applied on the two parallel lines. By using the expression of $\KstarOm$ on $L^2_0(\p \Om)$ in terms of the Fourier transform, we define a Hilbert space (denoted by $K^{-1/2}_0$), which is analogous to $\mathcal{H}^*$ of the Lipschitz domain case, and generalize $\KstarOm$ on $L^2_0(\p \Om)$ to $K^{-1/2}_0$ such that $\KstarOm$ is bounded, self-adjoint. As a main result, we derive the complete spectral resolution of the NP operator on the space $K^{-1/2}_0$. It turns out that the spectrums of the NP operators on a crescent--shaped domain and touching disks are both 
\begin{equation*}
  \sigma(\KstarOmega)=\sigma_{ac}(\KstarOmega)=[-1/2,1/2].
\end{equation*}
Note that this is identical to the limit of the spectrum of the separated disks in \eqnref{eqn:twosepdiskseig} and that of intersecting disks in \eqnref{eqn:intersectingdisksspect} as $\delta$ and $\theta$ tend to zero, respectively. 
The analyses for the crescent-shaped domain and touching disks go similarly.  
We provide the full details for the crescent-shaped domain and briefly state the results for touching disks without a detailed proof. 

As an application of the spectral resolution of the NP operator, we compute the order of plasmon resonance on a crescent-shaped domain. It is worth remarking that the plasmon resonance for the crescent-shaped domain was studied in \cite{Aubry:2010:BPD} but without mathematical rigor.

The remainder of this paper is organized as follows. In Section 2, we briefly explain the properties of the layer potential operators on a Lipschitz domain and plasmon resonance. Section 3, Section 4, and Section 5 address the crescent-shaped domain case. 
Section 3 is devoted to deriving expressions of the layer potential operators by using a M\"{o}bius transformation and the Fourier transform. In Section 4, we define the Hilbert space $K^{-1/2}_0$ and derive the spectral resolution of the NP operator. We then analyze the plasmon resonance in Section 5. In Section 6, we derive the spectral resolution of the NP operator for touching disks.

\section{Preliminary: layer potential operators on a Lipschitz domain}\label{subsection:Lipschitzver}
Let $D$ be a simply connected bounded Lipschitz domain in $\mathbb{R}^2$. 
For a density function $\psi\in \LtwoD$, the single-layer potential $\SingleD[\psi]$ is defined by 
\begin{align*}
  \SingleD[\psi](\bfx)&=\int_{\partial D}\Gamma(\bfx-\bfy)\psi(\bfy)\, d\sigma(\bfy),\quad \bfx\in\Rtwo,
\end{align*}
where $\Gamma(\bfx)$ is the fundamental solution to the Laplacian, that is $\Gamma(\bfx) = \frac{1}{2\pi}\ln|\bfx|$. The single-layer potential is harmonic in $\Rtwo\setminus\partial D$ and satisfies the jump relations \cite{Verchota:1984:LPR}:
\beq\label{eqn:Kstarjump}
\begin{aligned}
\SingleD[\psi]\Big|^{+}(\bfx)&=\SingleD[\psi]\Big|^{-}(\bfx)\quad\mbox{a.e. }\bfx\in\partial D,\\
\frac{\partial}{\partial\nu}\SingleD[\psi]\Big|^{\pm}(\bfx)&=\Big(\pm\frac{1}{2}I+\KstarD\Big)\psi(\bfx)\quad\mbox{a.e. }\bfx\in\partial D,
\end{aligned}
\eeq
where the NP operator $\KstarD$ is given by \eqnref{kstar:def}, the symbol $+$ and $-$ stand for the limit to $\partial D$ from the outside and inside from $\partial\Omega$, respectively, and $p.v$ denotes the Cauchy principal value.

The NP operator $\KstarD$ is in general not self-adjoint on $L^2(\partial D)$; however, it can be symmetrized using Plemelj's symmetrization principle (see \cite{Khavinson:2007:PVP}):
\begin{equation}\label{calderonidentity}
  \SingleD\KstarD = \KD\SingleD,
\end{equation}
where $\KD$ is the $L^2$ adjoint of $\KstarD$. We denote by $\mathcal{H}^*$ the space $H^{-1/2}_0(\partial D)$ equipped with the inner product 
\begin{equation}\label{def:Hstarnorm}
  \langle\varphi,\psi \rangle_{\mathcal{H}^*}:=-\int_{\partial D}\varphi \, {\SingleOmega[\psi]}\, d\sigma\quad\mbox{for }\varphi,\psi\in  H_{0}^{-1/2}({\partial D})
\end{equation}
and let $\lVert\cdot\rVert_{\mathcal{H}^*}$ be the corresponding norm, which is equivalent to the $H^{-1/2}$ norm, i.e.,
\begin{equation*}
  C_1\lVert\varphi\rVert_{H^{-1/2}}\leq
     \lVert\varphi\rVert_{\mathcal{H}^{*}}  \leq 
  C_2\lVert\varphi\rVert_{H^{-1/2}}
\end{equation*}
for all $\varphi\in H^{-1/2}_0(\partial D)$ with some positive constants $C_1$ and $C_2$.  From \eqnref{calderonidentity}, $\KstarD$ is self-adjoint on $\mathcal{H}^*$. 
As a result, the spectrum of $\KstarD$ lies on the real axis. In fact, the spectrum of $\KstarD$ on $\mathcal{H}^*$ lies in ${(-1/2,1/2)}$ \cite{Escauriaza:1992:RTW,Verchota:1984:LPR}. 

When the domain $D$ is the disk of radius $r_0>0$ centered at $\bfc$, it holds that
\begin{equation}\label{NP:Disk}
  \KstarD[\varphi] = \frac{1}{4\pi r_0}\int_{\partial D}\varphi \, d\sigma\quad\mbox{on }\p D
\end{equation}
and that the spectrum of $\KstarD$ consists of eigenvalues $0$ and $\frac{1}{2}$. The eigenspace that corresponds to the eigenvalue $0$ is $\Hstar(\p D)$, and a constant function is an eigenfunction of $\KstarD$ corresponding to the eigenvalue $\frac{1}{2}$. 
The single-layer potential for the constant function $\frac{1}{r_0}$ is
\begin{equation}\label{Scal:Disk}
  \SingleD\Big[\frac{1}{r_0}\Big](\bfx) = 
  \begin{cases}
    \ln r_0\quad&\text{if } \bfx\in D,\\
    \ln |\bfx-\bfc|\quad&\text{if } \bfx\in \mathbb{R}^2\setminus D.
  \end{cases}
\end{equation}

Suppose that the domain $D$ is occupied by a homogeneous material with the dielectric constant $\epsilon_c+i\delta$ ($\delta$ is the dissipation factor) and that the matrix $\mathbb{R}^2\setminus \overline{D}$ has the dielectric constant $\epsilon_m$. We assume that $\epsilon_m=1$. We express the dielectric constant of the entire space as
\begin{equation}\label{def:epsilon:D}
  \epsilon = (\epsilon_c+i\delta)\chi_D +\chi_{\mathbb{R}^2\setminus\overline{D}},
\end{equation}
where $\chi_A$ means the characteristic function of a set $A$. We now consider the potential problem
\begin{equation}\label{eqn:u_delta}
  \begin{cases}
  \nabla\cdot\epsilon\nabla u_\delta = f &\text{in }\mathbb{R}^2,\\
  u_\delta(\bfz) = O(|\bfz|^{-1}) &\text{as }|\bfz|\rightarrow \infty,
  \end{cases}
\end{equation}
where $f$ is a source function that is compactly supported in $\mathbb{R}^2\setminus\overline{D}$ satisfying $\int_{\RR^2} f d\bfz=0$. An example of the source function is a polarized dipole $f(\bfz)=\bfp\cdot \nabla \delta_\bfq(\bfz)$, where $\delta_\bfq$ is the Dirac mass at $\bfq$, and $\bfp$, $\bfq$ are constant vectors. 
The solution $u_\delta$ can be expressed as
\begin{equation}\notag%\label{eqn:solutionrepresentationintro}
  u_\delta = F+ \SingleD[\varphi_\delta]\quad\mbox{in }\RR^2,
\end{equation}
where $F$ denotes the Newtonian potential of $f$, i.e.,
$F(\bfz)=\int_{\RR^2}\Gamma(\bfz-\bfy)f(\bfy)\,d\bfy,$
and the density function $\varphi_\delta$ satisfies
\begin{equation}\label{eqn:intgralequationintro}
  (\lambda_\delta I -\KstarD)[\varphi_\delta] = \partial_\nu F\quad\text{on }\partial D
\end{equation}
with $\lambda_\delta$ given by \begin{equation}\label{eqn:lambda}
  \lambda_\delta=\frac{\epsilon_c+1+i\delta}{2(\epsilon_c-1)+2i\delta}=\frac{\epsilon_c +1}{2(\epsilon_c -1)}+O(\delta).
\end{equation}
The plasmon resonance
\begin{equation}\label{eqn:resonanceconditionintro}
  \big\lVert\nabla u_\delta \big\rVert_{L^2(D)} \rightarrow \infty \quad\text{as } \delta \rightarrow 0,
\end{equation}
may occur depending on $\epsilon_c$ and the spectrum of $\KstarD$.
The blow-up rate (or, the resonance rate) in \eqnref{eqn:resonanceconditionintro} is essentially related to the blow-up feature of the norm of the density function $\varphi_\delta$ (see, for instance, \cite[Section 5]{Kang:2017:SRN} and \eqnref{S:varphi:upper} in Section \ref{subsection:plasmonicresonance}). 

One can classify the spectrum of the NP operator on a Lipschitz domain by the resonance rate \cite{Helsing:2017:CSN}.
For $g\in \mathcal{H}^*(\p D)$, let $\varphi_{t,\delta}$ be the solution to 
\begin{equation*}
  ((t+i\delta)I-\mathcal{K}^*_{\p D})[\varphi_{t,\delta}]=g\quad\mbox{on }\p D
\end{equation*}
and define an indicator function 
\begin{equation*}
  \alpha_g(t):=\sup\{\alpha \mid \limsup_{\delta\rightarrow0} \delta^\alpha\lVert\varphi_{t,\delta}\rVert_{\mathcal{H}^*}=\infty\},\quad t\in(-1/2,1/2).
\end{equation*}
Then, $0\leq \alpha_g(t)\leq 1$ for all $t$. \begin{theorem}[\cite{Helsing:2017:CSN}]
Let $g\in\mathcal{H}^*$. For $t\in(-1/2,1/2)$, the following holds. 
\begin{enumerate}[(a), left=0.5em]
  \item If $\alpha_g(t)>0$, then $t\in \sigma(\mathcal{K}^*_{\p D})$.
  \item If $\alpha_g(t)=1$ and $t$ is isolated, then $t\in \sigma_{pp}(\mathcal{K}^*_{\p D})$.
  \item If $\frac{1}{2}\leq \alpha_g(t)<1$, then $t\in \sigma(\mathcal{K}^*_{\p D})$.
\end{enumerate}
\end{theorem}

\section{Layer potential operators on a crescent-shaped domain}\label{section:crescentdomain}

We consider a crescent-shaped domain $\Omega$ that is the region enclosed by the boundaries of two touching disks $B_R$ and $B_r$ such that $B_r\subset B_R$. In other words,
$\Omega = B_R\setminus\overline {B_r}$. The boundary of $\Om$ is composed of two circles $\partial B_R$ and $\partial B_r$ that are tangent at the origin point; see the left figure in Figure \ref{fig:domains}.  As $\Om$ has a cusp on its boundary, one cannot apply the results of the layer potential operators of Lipschitz domains. Instead, we will generalize the concepts of the single-layer potential and the NP operator to the crescent-shaped domain by using a M\"{o}bius transformation. We will then derive the integral expressions of the layer potential operators via the Fourier transform on $\RR$. 
\begin{figure}[b!]
\centering
\subfloat[$z$-plane, $z=z_1+iz_2$]{\includegraphics[height=4.5cm, width=5.72cm, trim={3.6cm 2cm 2cm 2cm}, clip]{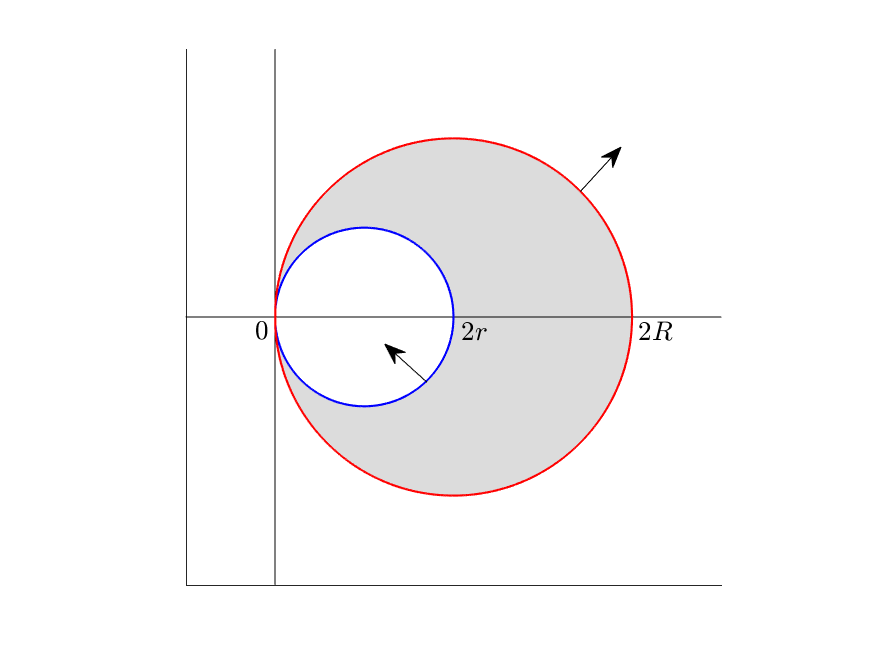}}
\qquad
\subfloat[$w$-plane, $w=x+iy$]{\includegraphics[height=4.5cm, width=5.72cm, trim={3.6cm 2cm 2cm 2cm}, clip]{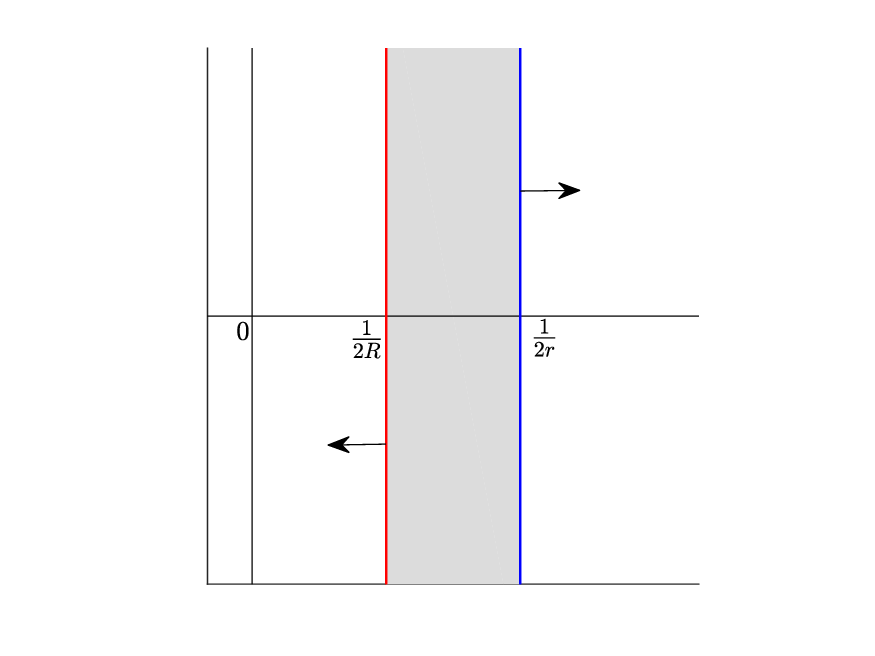}}
\caption{A crescent-shaped domain $\Omega=B_R\setminus\overline B_r$ (gray region in the left figure) and the vertical strip $S=\Psi(\Omega)$ (gray region in the right figure). Arrows indicate the outward normal vectors to $\p \Om$ or $\p S$.}
\label{fig:domains}
\end{figure}

\subsection{M\"{o}bius transformation}
We identify $\bfz=(z_1,z_2)\in\RR^2$ with $z=z_1+iz_2\in\CC$. 
Let $\Psi:\mathbb{C}\setminus{\{0\}}\rightarrow \mathbb{C}\setminus\{0\}$ be the M\"{o}bius transformation, that is,
\begin{equation}\label{eqn:Mobius}
  w=\Psi(z) = \frac{1}{z}.
\end{equation}
Obviously, $\Psi$ is a conformal mapping.
%We identify $(x,y)\in\RR^2$ with $x+iy\in\CC$ and set $w = x + i y$. 
We set $w=x+iy$, $(x,y)\in\RR^2$. Note that $\Psi^{-1}(w)=\Psi(w)=\frac{1}{w}$. The scale factors of the mapping $w\longmapsto \Psi^{-1}(w)$ with respect to $x$ and $y$ coincide. We denote the scale factor by 
\begin{equation*}
   h(x,y)=|\Psi'(w)|=\frac{1}{x^2+y^2}.
\end{equation*}

The M\"{o}bius transformation $\Psi$ maps the left half-plane to the left half-plane, and maps the right half-plane to the right half-plane. In particular, $\Psi$ maps a disk of radius $|a|$ centered at $(a,0)$ to the half-plane determined by $x> \frac{1}{2a}$ if $a>0$, and to the half plane determined by $x<\frac{1}{2a}$ if $a<0$.
We set for $a\neq0\in\RR$ that 
\begin{align}\label{def:h_a}
B_a=\big\{(x,y)\in\RR^2\mid|x+iy-a|\leq |a|\big\}, \quad
  h_a(y)=h\Big(\frac{1}{2a},y\Big)=\frac{1}{\big(\frac{1}{2a}\big)^2+y^2}.
\end{align}
The disks $B_R$ and $B_r$ are defined in this sense. The crescent-shaped domain $\Omega=B_R\setminus\overline{B_r}$ is mapped via the M\"{o}bius transformation onto the vertical strip (see Figure \ref{fig:domains})
\begin{equation*}
 S:= \Psi(\Om) = \Big(\frac{1}{2R},\frac{1}{2r}\Big)\times
  (-\infty,\infty),
\end{equation*}
and vice versa. 
For later use, we denote by $q$ the width of $S$, i.e., 
\beq\label{def:q}
q = \frac{1}{2r}-\frac{1}{2R}>0.
\eeq

We denote by $\nu$ the outward unit normal vector to $\p\Omega$, except at the touching point of $\p B_R$ and $\p B_r$. Note that $\nu$ is directed toward the exterior of $B_R$ on $\partial B_R$, but toward the interior of $B_r$ on $\partial B_r$; see  Figure \ref{fig:domains}. Following our normal vector convention, we then define the normal derivative of a function $v$: at $z\in\p\Om$ with $\Psi(z)=x+iy$,
\beq \label{normal_deriv}
 \begin{aligned}
  \frac{\partial v}{\partial\nu}&=-\frac{1}{h_R(y)}\frac{\partial(u\circ\Psi)}{\partial x}\quad\text{for } z\in\partial B_R\setminus\{0\},\\
  \frac{\partial v}{\partial\nu}&=\frac{1}{h_r(y)}\frac{\partial(u\circ\Psi)}{\partial x}\quad\text{for }z\in\partial R_r\setminus\{0\}.
\end{aligned}
\eeq

Recall that the value of an integrand function at a set of measure zero doesn't affect the integral value. 
We disregard the origin point, where two normal vectors are defined, in the layer potential formulation on the crescent-shaped domain $\Om$ in the following subsection. 

\subsection{Generalization of the layer potential operators to the crescent-shaped domain}\label{subsection:singleandNP}
A density function $\varphi\in\LtwobdOmega$ can be decomposed as
\begin{align}
   \varphi &= \varphi\,\chi_{\p B_R}+\varphi\,\chi_{\partial B_r} \notag \\ \label{def:varphiR_r}
&=:\varphi_R +\varphi_r.
\end{align}
The normal vector $\nu$ of $\p\Om$ points toward the exterior of $B_R$ on $\p B_R$ and toward the interior of $B_r$ on $\p B_r$, as mentioned before.
We define $\Kcal^*_{\p B_r}$ and $\Kcal^*_{\p B_R}$ by the integral expression \eqnref{kstar:def} with this normal vector convention.
We now define the single-layer potential and the NP operator on the crescent-shaped domain as follows. 
\begin{definition}
For $\varphi=\varphi_R+\varphi_r\in L^2(\p\Om)$, we define 
\begin{align}\label{def:SignleOmegadef}
  \SingleOmega[\varphi]
   :=\SingleBR[\varphi_R]
   + \SingleBr[\varphi_r]\quad\mbox{on }\RR^2
   \end{align}
and
\beq\label{def:NPoperatorcrescent1}
\begin{aligned}
  \KstarOmega[\varphi]
  :=&\Big(\KstarBR[\varphi_R] + \frac{\partial}{\partial\nu}\SingleBr[\varphi_r]\Big|_{\partial B_R}\Big)\chipBR\\
 &+\Big(\frac{\partial}{\partial\nu}\SingleBR[\varphi_R]\Big|_{\partial B_r} + \KstarBr[\varphi_r]\Big)\chi_{\partial B_r} \quad\text{on } \partial\Omega.
%  &=\KstarOmega[\psi]\chipBR +\KstarOmega[\psi]\chi_{\partial B_r} \quad a.e.,\quad\text{ on } \partial\Omega,\nonumber
\end{aligned}
\eeq
\end{definition}

%Then, $\SingleOmega$ satisfies the same jump relation as \eqnref{eqn:Kstarjump} as follows. 
\begin{lemma}\label{SK:relations} We have
\begin{align*}
\SingleOmega[\psi]\Big|^{+}(\bfx)&=\SingleOmega[\psi]\Big|^{-}(\bfx)\quad\mbox{a.e. }\bfx\in\partial \Omega,\\
  \frac{\partial}{\partial\nu}\SingleOmega[\varphi]\Big|^\pm_{\p\Om}(\bfx)
  &=
  \Big(\pm\frac{1}{2}I+\KstarOmega \Big)[\varphi](\bfx)\quad \mbox{a.e. } \bfx\in\partial\Omega.
\end{align*}
\end{lemma}
\begin{proof}
The continuity of the single-layer potential $\Scal_{\p\Om}[\varphi]$ across $\p\Om$ directly follows from the continuity of $\Scal_{\p B_R}[\varphi_R]$ and $\Scal_{\p B_r}[\varphi_r]$. 

On the boundary circles $\partial B_R$ and $\partial B_r$, we can apply the results of the layer potential operators on Lipschitz domains that are described in Subsection \ref{subsection:Lipschitzver}. Applying the jump relation \eqnref{eqn:Kstarjump}, we have
\begin{align*}
\frac{\partial}{\partial\nu}\SingleBR[\varphi_R]\Big|^\pm_{\partial B_R} 
&=\Big(\pm\frac{1}{2}I+\KstarBR\Big)[\varphi_R]\quad\mbox{on }\p B_R,\\
\frac{\partial}{\partial\nu}\SingleBr[\varphi_r]\Big|^\pm_{\partial B_r} 
&=\Big(\pm\frac{1}{2}I+\KstarBr\Big)[\varphi_r]\quad\mbox{on }\p B_r,
\end{align*}
where $\nu$ is the outward normal vector to $\p\Om$, the interior and exterior limits are defined corresponding to the direction of $\nu$, and $\Kcal^*_{\p B_r}$ and $\Kcal^*_{\p B_R}$ are also defined with this normal vector convention.
Also, we have
\begin{equation}\notag%\label{eqn:SBrBRnu}
  \frac{\partial}{\partial\nu}\SingleBr[\varphi_r]\Big|^\pm_{\partial B_R}
  =
  \begin{dcases}
  \frac{\partial}{\partial\nu}\SingleBr[\varphi_r]\Big|_{\partial B_R} \quad&\mbox{on } \p B_R\setminus \{0\},\\
   \Big(\pm\frac{1}{2}I+\KstarBr\Big)[\varphi_r]\quad&\text{at } 0
  \end{dcases}
\end{equation}
and
\begin{equation}\notag%\label{eqn:SBRBrnu}
  \frac{\partial}{\partial\nu}\SingleBR[\varphi_R]\Big|^\pm_{\partial B_r} 
  =
  \begin{dcases}
    \frac{\partial}{\partial\nu}\SingleBR[\varphi_R]\Big|_{\partial B_r} \quad&\text{on } \p B_r\setminus \{0\},\\
       \Big(\pm\frac{1}{2}I+\KstarBR\Big)[\varphi_R]\quad&\text{at } 0.
  \end{dcases}
\end{equation}
Hence, we prove the lemma.
\end{proof}
We emphasize that it is necessary to analyze the mapping properties of $\frac{\partial}{\partial\nu}\SingleBr[\varphi_r]\big|_{\partial B_R}$ and $\frac{\partial}{\partial\nu}\SingleBR[\varphi_R]\big|_{\partial B_r}$ to understand the spectral structure of the NP operator on the crescent-shaped domain.
% \smallskip

As in the previous subsection, we set 
\beq\label{z:y}
z=\frac{1}{x+iy}\quad\mbox{for }z\neq 0.
\eeq
Let $z_t=\Psi^{-1}({\frac{1}{2r}+it})$ on $\p B_r$ and $z_t=\Psi^{-1}(\frac{1}{2R}+it)$ on $\p B_R$; then we have
$|z-z_t|=\frac{| x-\frac{1}{2r} +i(y-t)|}{|x+iy||\frac{1}{2r}+it|}$ on $\p B_r$ and a similar relation on $\p B_R$, respectively. We identify $\varphi_R$, $\varphi_r$ in \eqnref{def:varphiR_r} with the functions on $\RR$ given by
\beq \label{varphi:tilde}
\begin{aligned}
\widetilde{\varphi}_R(y)&=(\varphi_R\circ\Psi^{-1})\Big(\frac{1}{2R}+iy\Big),\\
\widetilde{\varphi}_r(y)&=(\varphi_r\circ\Psi^{-1})\Big(\frac{1}{2r}+iy\Big) \quad\mbox{for }y\in\RR.
\end{aligned}
\eeq
Then, the single-layer potential \eqnref{def:SignleOmegadef} satisfies
\begin{equation}\label{eqn:singleOmegaLtwo}
\begin{aligned}
  \SingleOmega[\varphi](z) 
 &=\frac{1}{2\pi}\int_{\partial B_r}\ln|z-z_t|\varphi_r(z_t)\, d\sigma (z_t)+\frac{1}{2\pi}\int_{\partial B_R}\ln|z-z_t|\varphi_R(z_t)\, d\sigma (z_t)\\
   &=\frac{1}{4\pi}\int_{-\infty}^\infty 
  \bigg(\ln\Big[\Big(x-\frac{1}{2r}\Big)^2+(y-t)^2\Big]-\ln\Big[\Big(\frac{1}{2r}\Big)^2+t^2\Big]\bigg)\tvarphi_r(t)h_r(t)\, dt\\
 &\hskip .2cm +\frac{1}{4\pi}\int_{-\infty}^\infty 
  \bigg(\ln\Big[\Big(x-\frac{1}{2R}\Big)^2+(y-t)^2\Big]-\ln\Big[\Big(\frac{1}{2R}\Big)^2+t^2\Big]\bigg)\tvarphi_R(t) h_R(t)\, dt\\
 &\hskip .2cm -\frac{1}{4\pi}\ln(x^2+y^2)\int_{\partial\Omega}\varphi(z)\, d\sigma(z)
\end{aligned}
\end{equation}
and
\begin{align*}
\frac{\partial}{\partial\nu}\SingleBr[\varphi_r]\Big|_{\partial B_R}(z)
	&=\frac{1}{2\pi}\int_{\partial B_r}\frac{\partial}{\partial\nu_z}\ln|z-z_t|\varphi_r(z_t) \, d\sigma(z_t)\nonumber\\
	&=\frac{1}{2\pi}\frac{1}{ h_R(y)}
      \int_{-\infty}^\infty 
      \bigg(\frac{q}{q^2+(y-t)^2}+\frac{\frac 1 {2R}}{(\frac 1 {2R})^2+y^2}\bigg)\tvarphi_r(t) h_r(t)\, dt\nonumber\\
   &=\frac{1}{2\pi}\frac{1}{h_R(y)}
      \int_{-\infty}^\infty \frac{q}{q^2+(y-t)^2}\, \tvarphi_r(t)h_r(t)\, dt+\frac{1}{4\pi R}\int_{-\infty}^\infty\tvarphi_r(t)h_r(t)\, dt,\\
%      \end{align*}
%      and
%      \begin{align*}
\frac{\partial}{\partial\nu}\SingleBR[\varphi_R]\Big|_{\partial B_r}(z)
    &=\frac{1}{2\pi}\frac{1}{ h_r(y)}
      \int_{-\infty}^\infty \frac{q}{q^2+(y-t)^2} \,\tvarphi_R(t)h_R(t)\, dt-\frac{1}{4\pi r}\int_{-\infty}^\infty\tvarphi_R(t)h_R(t)\, dt,
\end{align*}
where $q$ denotes the width of the strip $\Psi(\Om)$ (that is, $q=\frac{1}{2r}-\frac{1}{2R}$).
Furthermore, it holds from \eqnref{NP:Disk} that
 \begin{align}\notag
  \KstarBR[\varphi_R]&=\frac{1}{4\pi R}\int_{\partial B_R}\varphi_R \, d\sigma\quad\mbox{on }\p B_R,\\\notag
  \KstarBr[\varphi_r]&=-\frac{1}{4\pi r}\int_{\partial B_r}\varphi_r \, d\sigma\quad\mbox{on }\p B_r.
\end{align}
Hence, for $\varphi=\varphi_R+\varphi_r\in L^2(\p\Om)$, it holds that
%\begin{lemma} For $\varphi$ given by \eqnref{varphi:stack}, it holds that
\beq\label{eqn:Kstarbasicresult}
 \KstarOmega[\varphi](z)
  =
  \begin{dcases}
\frac{1}{2\pi h_R(y)}
      \int_{-\infty}^\infty \frac{q}{q^2+(y-t)^2} \, \tvarphi_r(t)\, h_r(t)\, dt +\frac{1}{4\pi R}\int_{\partial\Omega}\varphi\, d\sigma
    \quad&\mbox{for }x=\frac{1}{2R},\\
\frac{1}{2\pi h_r(y)}
      \int_{-\infty}^\infty \frac{q}{q^2+(y-t)^2}\, \tvarphi_R(t)\,h_R(t)\, dt-\frac{1}{4\pi r}\int_{\partial\Omega}\varphi \, d\sigma \quad&\mbox{for }x=\frac{1}{2r}
\end{dcases}
\eeq
with $\tvarphi_R$ and $\tvarphi_r$ given by \eqnref{varphi:tilde}.
 
%\end{lemma}
\subsection{Layer potential operators in terms of the Fourier transform}\label{subsec:layer:Fourier}
%\subsection{Layer potential operators in terms of the Fourier transform}
The Fourier transform and its inversion in $\RR$ are defined as
\begin{align*}
  \mathcal{F}[f](k)&= \frac{1}{\sqrt{2\pi}}\int_{-\infty}^\infty f(y)e^{-iky}\, dy,\\
  \mathcal{F}^{-1}[f](k)&= \frac{1}{\sqrt{2\pi}}\int_{-\infty}^\infty f(y)e^{iky}\, dy.
\end{align*}
Recall that we identify $\varphi\in L^2(\p\Om)$ with two functions $\widetilde \varphi_R$, $\widetilde \varphi_r$ given by \eqnref{varphi:tilde}. We can further identify $\varphi$, via the Fourier transform, with
\begin{align}\label{def:U}
  U[\varphi]
&:=
  \begin{bmatrix}
    \mathcal{F}[h_R \widetilde{\varphi}_R]\\[2mm]
    \mathcal{F}[h_r \widetilde{\varphi}_r]
   \end{bmatrix}.
   \end{align}
The inversion of the operator $U$ is
\begin{align*}
   U^{-1}\begin{bmatrix}
     {f}_R\\[2mm]
     {f}_r
   \end{bmatrix}(z)
   =
   \begin{dcases}
    \frac{1}{h_R(y)}\mathcal{F}^{-1}[{f}_R](y)\quad&\mbox{for }z\in\p B_R,\\
    \frac{1}{h_r(y)}\mathcal{F}^{-1}[{f}_r](y)\quad&\mbox{for }z\in\p B_r,
   \end{dcases}
\end{align*}
where $z$ and $y$ satisfy the relation \eqnref{z:y}. 

We now express the single-layer potential and the NP operator for $\varphi$ in terms of $U[\varphi]$ as follows. 
\begin{lemma}\label{lemma:SinglelayeronclassM}
Let $\varphi = \varphi_R+\varphi_r\in \LtwozerobdOmega$. For $z=\Psi^{-1}(x+iy)\in\p\Om$, we have
\beq\notag
 \begin{aligned}
   \SingleOmega[\varphi](z)
   = -\frac{1}{\sqrt{2\pi}}\int_{-\infty}^\infty\frac{1}{2|k|}
  \Big(e^{-|x-\frac{1}{2R}||k|}\mathcal{F}[h_R\widetilde{\varphi}_R](k)
   +e^{-|x-\frac{1}{2r}||k|}\mathcal{F}[h_r\widetilde{\varphi}_r](k) \Big) e^{iky}\, dk+C,  \end{aligned}
\eeq
where $C$ is the constant given by 
\begin{equation*}
  C=\frac{1}{\sqrt{2\pi}}\int_{-\infty}^\infty\frac{1}{2|k|}
    \Big(e^{-\frac{1}{2R}|k|}\mathcal{F}[h_R\widetilde{\varphi}_R](k)
     +e^{-\frac{1}{2r}|k|}\mathcal{F}[h_r\widetilde{\varphi}_r](k) \Big) dk .
\end{equation*}
% $C$ is a constant such that $\Scal_{\p\Om}[\varphi](z)\rightarrow0$ as $|z|\rightarrow\infty$ (i.e., as $|x+iy|\rightarrow 0$).
\end{lemma}
\begin{proof}
The assumption $\varphi\in\LtwozerobdOmega$ implies that $h_R|\widetilde{\varphi}_R|^2,\, h_r|\widetilde{\varphi}_r|^2 \in L^1{(\mathbb{R})}$.
Since $h_R$ and $h_r$ are bounded and integrable on $\RR$, we have
\beq\label{hrvarphi:L1L2}
h_R\widetilde{\varphi}_R ,h_r\widetilde{\varphi}_r \in L^1{(\mathbb{R})}\cap L^2{(\mathbb{R})}.
\eeq

Note that for fixed $a,b,A,B$, the function $\ln(A^2+(a-t)^2)-\ln(B^2+(b-t)^2)$ is square integrable on any bounded interval of $t$. Furthermore, we have 
\beq\label{eqn:logdecayatinfinity}
\begin{aligned}
&\ln\big(A^2+(a-t)^2\big)-\ln\big(B^2+(b-t)^2\big)\\
  =\, &(-2 a + 2 b)\frac{1}{t} + (-a^2 + A^2 + b^2 - B^2)\frac{1}{t^2}+O\Big(\frac{1}{t^3}\Big)\quad\text{as } t\rightarrow \infty,
\end{aligned}
\eeq
where $O(\frac{1}{t^3})$ is uniformly bounded with respect to small $|b|\geq 0$ (with fixed $a,A,B$). 
Applying the dominated convergence theorem to \eqnref{eqn:singleOmegaLtwo}, we obtain
\begin{align*}
  &\SingleOmega[\varphi](\Psi^{-1}(x+iy)) \\
  =&\lim_{c\rightarrow y}
  \Bigg[\frac{1}{4\pi}\int_{-\infty}^\infty 
  \bigg(\ln\Big[\Big(x-\frac{1}{2R}\Big)^2+(y-t)^2\Big]-\ln\Big[\Big(\frac{1}{2R}\Big)^2+(t-y+c)^2\Big]\bigg)\,\widetilde \varphi_R(t)h_R(t)\, dt\\
  &\quad +\frac{1}{4\pi}\int_{-\infty}^\infty 
  \bigg(\ln\Big[\Big(x-\frac{1}{2r}\Big)^2+(y-t)^2\Big]-\ln\Big[\Big(\frac{1}{2r}\Big)^2+(t-y+c)^2\Big]\bigg)\,\widetilde \varphi_r(t)h_r(t)\, dt\Bigg].
\end{align*}
The last term in \eqnref{eqn:singleOmegaLtwo} vanishes assuming the mean-zero condition on $\varphi$. 
The Fourier transform of $\ln(y^2+a^2)$ is 
\beq\label{fourier:ln}
\mathcal{F}[\ln(y^2+a^2)](k)=-\sqrt{2\pi}\bigg(\frac{e^{-|a||k|}}{|k|}+2\gamma_E\delta(k)\bigg),
\eeq
where $\frac{1}{|k|}$ is defined in the sense of principal value and $\gamma_E$ denotes Euler's constant.
The convolution theorem of the Fourier transform, i.e., $\mathcal{F}[f_1*f_2]=\sqrt{2\pi}\mathcal{F}[f_1]\,\mathcal{F}[f_2]$, leads to the relation
\beq\label{Scal:step1}
\begin{aligned}
  \SingleOmega[\varphi](z)
  =&\lim_{c\rightarrow y}\frac{-1}{\sqrt{2\pi}}\Bigg[\,
    \int_{-\infty}^\infty\frac{1}{2|k|}\Big(e^{-|x-\frac{1}{2R}||k|}-e^{-\frac{|k|}{2R}}e^{-ikc}\Big)\mathcal{F}[h_R \widetilde\varphi_R](k) e^{iky}\, dk\\
   &\quad\quad \quad+\int_{-\infty}^\infty\frac{1}{2|k|}\Big(e^{-|x-\frac{1}{2r}||k|}-e^{-\frac{|k|}{2r}}e^{-ikc}\Big)\mathcal{F}[h_r \widetilde \varphi_r](k)e^{iky}\, dk\Bigg].
\end{aligned}
\eeq

From the mean-zero assumption on $\varphi$, we have
\beq\label{varphi:meanzero}
\int_{-\infty}^\infty\widetilde{\varphi}_Rh_R\, dy +\int_{-\infty}^\infty\widetilde{\varphi}_rh_r\, dy =0
\eeq
and, hence,
\begin{align*}
  &\mathcal{F}[h_R\widetilde{\varphi}_R](k)+\mathcal{F}[h_r\widetilde{\varphi}_r](k)\\
  =&\frac{1}{\sqrt{2\pi}}\int_{-\infty}^\infty(\widetilde{\varphi}_R h_R)(y)\big(e^{-iky}-1\big)dy
  +\frac{1}{\sqrt{2\pi}}\int_{-\infty}^\infty(\widetilde{\varphi}_r h_r)(y)\big(e^{-iky}-1\big)dy.
% =\, O(|k|^{\frac{1}{4}})\quad\text{as }|k|\rightarrow 0.
\end{align*}
From \eqnref{hrvarphi:L1L2} and the fact that $h_R$ is uniformly bounded, we have
\begin{align*}
\Big|\int_{-\infty}^\infty(\widetilde{\varphi}_R h_R)(y)\big(e^{-iky}-1\big)dy\Big|
\leq &\lVert\tvarphi_R {(h_R)}^{\frac 1 2}\rVert_{L^2(\RR)} \lVert{(h_R)}^{\frac 1 2}\big(e^{-iky}-1\big)\rVert_{L^2(\RR)}\\
\leq&\, C (\int_{|y|<\frac{1}{\sqrt{|k|}}}\frac{\big|e^{-iky}-1\big|^2}{\frac{1}{4R^2}+y^2} dy+\int_{|y|\geq\frac{1}{\sqrt{|k|}}}\frac{\big|e^{-iky}-1\big|^2}{\frac{1}{4R^2}+y^2}dy)^{\frac{1}{2}}\\
\leq &\, C(\int_{|y|<\frac{1}{\sqrt{|k|}}} \frac{|ky|^2}{\frac{1}{4R^2}} dy
+\int_{|y|\geq\frac{1}{\sqrt{|k|}}}\frac{1}{y^2}dy)^{\frac 1 2}\\
\leq &\, C|k|^{\frac{1}{4}}
\end{align*}
for some positive constant $C$, and a similar relation holds for $\tvarphi_r h_r$. Thus, for any constants $a$ and $b$, it holds that
\begin{equation}\label{Fourier:L2_0}
e^{-|a||k|}\mathcal{F}[h_R\widetilde{\varphi}_R](k)+e^{-|b||k|}\mathcal{F}[h_r\widetilde{\varphi}_r](k)=O(|k|^{\frac 1 4})  \quad\text{as }|k|\rightarrow 0.
\end{equation}
Then, we can apply the dominated convergence theorem to \eqnref{Scal:step1} and, as a result, change the order of the limit and integration. 
%Applying again the dominate convergence theorem to \eqnref{Scal:step1}, we obtain 
%\beq\notag
% \begin{aligned}
%   \SingleOmega[\varphi](z)
%   =&-\frac{1}{\sqrt{2\pi}}\int_{-\infty}^\infty\frac{1}{2|k|}(e^{-|x-\frac{1}{2R}||k|}-e^{-\frac{|k|}{2R}}e^{-iky})\mathcal{F}[h_R\widetilde{\varphi}_R](k)\, e^{iky}\, dk\\
%   &-\frac{1}{\sqrt{2\pi}}\int_{-\infty}^\infty\frac{1}{2|k|}(e^{-|x-\frac{1}{2r}||k|}-e^{-\frac{|k|}{2r}}e^{-iky})\mathcal{F}[h_r\widetilde{\varphi}_r](k)\, e^{iky}\, dk.
% \end{aligned}
% \eeq
Hence, we prove the lemma.
\end{proof}

Recall that $q=\frac{1}{2r} -\frac{1}{2R}>0$.
We set
\begin{equation}\label{P:def}
  P:=\frac{1}{\sqrt{2}}\begin{bmatrix}
    -1&1\\1&1
  \end{bmatrix}.
\end{equation}
This matrix satisfies $P=P^{-1}=P^T$ and, for any $d_1,d_2$, 
\beq\label{p:prop}
\frac{1}{2}
\begin{bmatrix}d_1+d_2&-d_1+d_2\\
-d_1+d_2&d_1+d_2
\end{bmatrix}
=P^{-1}
\begin{bmatrix}
d_1&0\\0&d_2
\end{bmatrix}
P.
\eeq

\begin{lemma}\label{lemma:KStarfreqency}
For $\varphi=\varphi_R+\varphi_r\in \LtwozerobdOmega$, it holds that
\begin{equation}\label{eqn:KStarfreqencydomain}
U[\KstarOmega[\varphi]]
   = P^{-1}\Big(\frac{1}{2}\, e^{-q|k|}
   \begin{bmatrix}
    -1&0\\0&1
   \end{bmatrix}
   P\Big)U[\varphi].
\end{equation} 
\end{lemma}
\begin{proof} 
%It then follows from \eqnref{eqn:Kstarbasicresult} that
%\beq\notag
% \KstarOmega[\varphi](z)
%  =
%  \begin{dcases}
%\frac{1}{2\pi h_R(y)}
%      \int_{-\infty}^\infty \frac{q}{q^2+(y-t)^2} \, \tvarphi_r(t) h_r(t)\, dt 
%    \quad&\mbox{on }\p B_R,\\
%\frac{1}{2\pi h_r(y)}
%      \int_{-\infty}^\infty \frac{q}{q^2+(y-t)^2}\, \tvarphi_R(t)h_R(t)\, dt
%      \quad&\mbox{on }\p B_r.
%\end{dcases}
%\eeq
It holds that
\begin{equation*}
  \mathcal{F}\bigg[\frac{q}{q^2+y^2}\bigg](k)=\sqrt{\frac{\pi}{2}}\, e^{-q|k|}.
\end{equation*}
From the assumption $\varphi=\varphi_R+\varphi_r\in \LtwozerobdOmega$, we have $\int_{\p\Om}\varphi\, d\sigma=0$. 
It then follows by applying the convolution theorem of the Fourier transform to \eqnref{eqn:Kstarbasicresult} that
%Hence, the convolution theorem of the Fourier transform leads to the relation
\begin{align*}
  U[\KstarOmega[\varphi]]
  = &\frac{1}{2}e^{-q|k|}
  \begin{bmatrix}
    0&1\\1&0
   \end{bmatrix}\begin{bmatrix}
    \mathcal{F}[h_R \widetilde{\varphi}_R]\\
    \mathcal{F}[h_r \widetilde{\varphi}_r]
   \end{bmatrix}  
  = \frac{1}{2}e^{-q|k|}
  \begin{bmatrix}
    0&1\\1&0
   \end{bmatrix}
  U[\varphi].
\end{align*}
From \eqnref{p:prop}, we obtain \eqnref{eqn:KStarfreqencydomain}.
\end{proof}

\begin{lemma}\label{lemma:Ltwozerocalderon}
For $\psi,\varphi\in\LtwozerobdOmega$, we have
  \begin{align}\notag%\label{eqn:calderonequiv}
   \int_{\partial\Omega}\psi(z)\, \overline{\SingleOmega[\varphi](z)}\, d\sigma(z)
   =\int_{-\infty}^\infty 
\big(U[\psi](k)\big)^T P^{-1}\bigg(-
   \frac{1}{2|k|}
   \begin{bmatrix}
     1-e^{-q|k|}&0\\
     0 &1+e^{-q|k|}
   \end{bmatrix}P\bigg)
   \overline{U[\varphi](k)}\, dk.
\end{align}
\end{lemma}
\begin{proof}
Set $\psi = \psi_R+\psi_r \text{ and } \varphi = \varphi_R+\varphi_r$.
From Lemma \ref{lemma:SinglelayeronclassM}, we have
\begin{align*}
  &\int_{\partial\Omega}\psi\, \overline{\SingleOmega[\varphi]}\, d\sigma\\
 =&\int_{-\infty}^\infty \tpsi_R(y)\, \overline{\Scal_{\p\Om}[\varphi]\Big(\frac{1}{1/(2R)+iy}\Big)}\,h_R(y)\, dy
 +\int_{-\infty}^\infty \tpsi_r(y)\,\overline{\Scal_{\p\Om}[\varphi]\Big(\frac{1}{1/(2r)+iy}\Big)}\,h_r(y)\, dy\\
  =&-\frac{1}{\sqrt{2\pi}}\int_{-\infty}^\infty \int_{-\infty}^\infty 
  (\tpsi_Rh_R)(y)\,\frac{1}{2|k|}\Big( \,\overline{\mathcal{F}[h_R\widetilde{\varphi}_R](k)}
  +e^{-q|k|}\,\overline{\mathcal{F}[h_r\widetilde{\varphi}_r](k)}\,\Big)e^{-iky}\,dk\,dy\\
  &-\frac{1}{\sqrt{2\pi}}\int_{-\infty}^\infty \int_{-\infty}^\infty 
(\tpsi_rh_r)(y)\,\frac{1}{2|k|}\Big(e^{-q|k|}\, \overline{\mathcal{F}[h_R\widetilde{\varphi}_R](k)}
  +\overline{\mathcal{F}[h_r\widetilde{\varphi}_r](k)}\,\Big)e^{-iky}\,dk\,dy.
  \end{align*}
  By changing the order of integrations, one obtains
\begin{align*}
  -\int_{\partial\Omega}\psi\,\overline{\SingleOmega[\varphi]}\, d\sigma \notag
  =\int_{-\infty}^\infty\frac{1}{2|k|}\,
  \big(U[\psi](k)\big)^T
   \begin{bmatrix}
     1&e^{-q|k|}\\
     e^{-q|k|} &1
   \end{bmatrix}
  \overline{ U[\varphi](k)}\, dk.
   \end{align*}
In view of \eqnref{p:prop}, this completes the proof. 
\end{proof}

\section{Spectral resolution of the NP operator on a crescent-shaped domain}
We denote the two matrix-valued functions in Lemma \ref{lemma:KStarfreqency} and Lemma \ref{lemma:Ltwozerocalderon} as follows:
%Let $\mathbb{S}$ and $\mathbb{K}$ be the matrix valued functions given by
\beq\label{def:matrix:S:K}
\begin{aligned}
  \mathbb{S}&= -\frac{1}{2|k|}\begin{bmatrix}
      1-e^{-q|k|}&0\\
     0 &1+e^{-q|k|}
   \end{bmatrix}
   ,\quad k\in\mathbb{R}\setminus\{0\},\\[2mm]
%\end{equation*}
%and
%\begin{equation*}
  \mathbb{K}&=\frac{1}{2}e^{-q|k|}
  \begin{bmatrix}
    -1&0\\0&1
   \end{bmatrix},\quad k\in \mathbb{R}.
\end{aligned}
\eeq
In terms of these matrix-valued functions, we define a Hilbert space $K^{-1/2}_0$ that extends $L^2_0(\p\Om)$. We then generalize the layer potential operators to be defined on $K^{-1/2}_0$ by using the integral expressions in Subsection \ref{subsec:layer:Fourier}. We finally obtain the spectral resolution of the NP operator on $K^{-1/2}_0$.

\subsection{Hilbert space $K^{-1/2}_0$}\label{subsection:newspaceNPoperator}

For $\varphi\in L^2_0(\p\Om)$, it holds that $\varphi = U^{-1}P\hvarphi$ with 
\begin{equation*}
   \hvarphi =  \frac{1}{\sqrt{2}}\begin{bmatrix}
      -1&1\\1&1
    \end{bmatrix}\begin{bmatrix}
      \mathcal{F}[h_R \widetilde{\varphi}_R]\\[2mm]
      \mathcal{F}[h_r \widetilde{\varphi}_r]
     \end{bmatrix}.
\end{equation*}
From \eqnref{Fourier:L2_0}, we have 
$ \int_{-\infty}^\infty
   \hvarphi^T
   (-\mathbb{S})~
   \overline{\hvarphi} \, dk<\infty$.
Based on these relations, we define a Hilbert space: 
\begin{definition}\label{def:K_0}
We define 
\beq\label{varphi_K}
K^{-1/2}_0 := \bigg\{
  \varphi=U^{-1}P\hvarphi \,\Big\vert\, \hvarphi=\begin{bmatrix}\widehat{\varphi}_1\\\widehat{\varphi}_2\end{bmatrix}
   \mbox{ satisfying } \int_{-\infty}^\infty
   \hvarphi^T
   (-\mathbb{S})~
   \overline{\hvarphi} \, dk<\infty
   \bigg\},
   \eeq  where $\hvarphi_1$ and $\hvarphi_2$ are measurable functions on $\RR$, and $P$ is given by \eqnref{P:def}. 
\end{definition}

The inverse Fourier transform $U^{-1}$ in \eqnref{varphi_K} is defined in the tempered distribution sense. Indeed, for any function $f$ on $\RR$ satisfying
	\beq\label{cond:f:decay}
	\int_{-\infty}^\infty\frac{1}{1+|k|}\, |f(k)|^2\, dk<\infty,
	\eeq
	it holds that, for any  $\psi$ in the Schwartz class $\mathcal{S}$, 
	\begin{align*}
		\bigg|\int_{-\infty}^\infty f\psi\, dk\bigg|
		&\leq \bigg(\int_{-\infty}^\infty\frac{1}{1+|k|}|f|^2\, dk\bigg)^{1/2} \bigg(\int_{-\infty}^\infty (1+|k|)|\psi|^2\, dk\bigg)^{1/2}\\
		&\leq C\bigg(\int_{-\infty}^\infty (1+|k|)^4|\psi|^2
		\frac{1}{(1+|k|)^2}\, dk\bigg)^{1/2}
		\\
		&\leq C\big\lVert\, (1+|k|)^2\, \psi(k) \big\rVert_{\infty}\\
		&\leq C\sum_{\alpha\leq2}\lVert\psi\rVert_{\alpha,0}<\infty,
	\end{align*}
	where $\lVert\phi\rVert_{\alpha,\beta}:= \sup_{x\in\mathbb{R}}|x^\alpha\partial^\beta\phi(x)|$. 
	Therefore, we have $f\in\mathcal{S}'(\mathbb{R})$, where $\mathcal{S}'(\mathbb{R})$ denotes the class of tempered distributions. 
	The Fourier transform and its inversion on $L^2(\RR)$ can be extended to $\mathcal{S}'(\mathbb{R})$, where the inversion formula still holds. Since $\widehat{\varphi}_1$ and $\widehat{\varphi}_2$ satisfy the decay condition \eqnref{cond:f:decay}, we can define the inverse transform for $P\widehat{\varphi}$.

	As the $(2,2)$-component of $\mathbb{S}$ blows up as $|k|^{-1}$ near $k=0$, the condition $ \int_{-\infty}^\infty
	\hvarphi^T
	(-\mathbb{S})~
	\overline{\hvarphi} \, dk<\infty$ implies decay of $\hvarphi_2$ near $k=0$; we highlight this property by adding the subscript $0$ in $K_0^{-1/2}$. 
	On the other hand, we add the superscript $-{1}/{2}$ in $K_0^{-1/2}$ since we define an inner product in a similar way to \eqnref{def:Hstarnorm} (see \eqnref{def:single} below).

It is straightforward to obtain the following. 
\begin{lemma}\label{lemma:L20:K}
We have
\beq\label{eqn:L20:K}
\LtwozerobdOmega\subset K^{-1/2}_0.
\eeq
\end{lemma}

 Note that $K^{-1/2}_0$ is defined in the weighted $L^2$ sense. 
We accordingly define the inner product so that $K^{-1/2}_0$ is complete. In other words, we equip this space with the inner product 
\begin{align}\label{def:single}
  \langle \psi,\varphi\rangle_{-1/2}:=&
  - \int_{-\infty}^\infty
   \begin{bmatrix}
        \widehat{\psi}_1\\
        \widehat{\psi}_2
      \end{bmatrix}^T
   \mathbb{S}~
   \overline{\begin{bmatrix}
         \widehat{\varphi}_1\\
         \widehat{\varphi}_2
       \end{bmatrix}}\, dk\\\notag
       =&\, \frac{1}{2}\int_{-\infty}^{\infty}\widehat \psi_1(k) \, \overline{\widehat \varphi_1(k)}\, \frac{1-e^{-q|k|}}{|k|}\, dk
       + \frac{1}{2}\int_{-\infty}^{\infty}\widehat \psi_2(k) \, \overline{\widehat \varphi_2(k)}\, \frac{1+e^{-q|k|}}{|k|}\, dk,
\end{align}
where $\psi,\varphi\in K^{-1/2}_0$ are given by 
\beq\label{varphi:hat}
\varphi = U^{-1}P\begin{bmatrix}
     \widehat{\varphi}_1\\\widehat{\varphi}_2
   \end{bmatrix},\quad
\psi = U^{-1}P\begin{bmatrix}
     \widehat{\psi}_1\\\widehat{\psi}_2
   \end{bmatrix}.
      \eeq
We denote the associated norm by 
\begin{equation}\label{def:K:norm}
  \lVert\varphi\rVert_{-1/2}:=(\langle\varphi,\varphi \rangle_{-1/2} )^{\frac{1}{2}}
  =
  \bigg(- \int_{-\infty}^\infty
   \begin{bmatrix}
        \widehat{\varphi}_1\\
        \widehat{\varphi}_2
      \end{bmatrix}^T
  \mathbb{S}~
   \overline{\begin{bmatrix}
         \widehat{\varphi}_1\\
         \widehat{\varphi}_2
       \end{bmatrix}}\, dk\bigg)^{\frac{1}{2}}.
\end{equation}
%In other words, $K^{-1/2}_0$ is a weighted $L^2$ space. Hence, it is complete. 

\subsection{The NP operator and the single-layer potential on $K^{-1/2}_0$}
We define the NP operator and the single-layer potential on $K^{-1/2}_0$ by extending the formulas in Lemma \ref{lemma:KStarfreqency} and Lemma \ref{lemma:SinglelayeronclassM}, respectively.
\begin{definition}\label{def:Scal:K}
Let $\varphi\in  K^{-1/2}_0$ be given by $\varphi=U^{-1}P\hvarphi$.
\begin{enumerate}[(a), left=0.5em]
\item
We define the NP operator $\KstarOmega:K^{-1/2}_0\rightarrow K^{-1/2}_0$ 
by
\begin{equation}\label{eqn:KstarOMonKmonehalf}
\KstarOmega[\varphi]:=U^{-1}P\mathbb{K}\hvarphi.
\end{equation}
\item
We define the single-layer potential of $\varphi$: for $z=\Psi^{-1}(x+iy)\in\CC$,
\beq \label{eqn:SinglelayeronKmonehalf}
 \begin{aligned}
  \SingleOmega[\varphi](z):= &-\frac{1}{\sqrt{2\pi}}\int_{-\infty}^\infty\frac{1}{2|k|}
   \Big(e^{-|x-\frac{1}{2R}||k|}\,\hvarphi_R(k)
   +e^{-|x-\frac{1}{2r}||k|}\,\hvarphi_r(k) \Big) e^{iky}\, dk \\
   &+\frac{1}{\sqrt{2\pi}}\int_{-\infty}^\infty\frac{1}{2|k|}
   \Big(e^{-\frac{|k|}{2R}}\,\hvarphi_R(k)+e^{-\frac{|k|}{2r}}\,\hvarphi_r(k)\Big)\, dk,
    \end{aligned}
   \eeq
where $\hvarphi_R$ and $\hvarphi_r$ is given by 
\beq\label{def:hvarphi_R_r}
 \begin{bmatrix}
   \hvarphi_R\\ 
   \hvarphi_r
   \end{bmatrix}
  =P
  \begin{bmatrix}
  \hvarphi_1\\
  \hvarphi_2
  \end{bmatrix}
  =\frac{1}{\sqrt{2}}
  \begin{bmatrix}
  -\hvarphi_1+\hvarphi_2\\
  \hvarphi_1+\hvarphi_2
  \end{bmatrix}.
  \eeq
  \end{enumerate}
\end{definition}

\smallskip

It is worth emphasizing that \eqnref{eqn:KstarOMonKmonehalf} and \eqnref{eqn:SinglelayeronKmonehalf} hold for $\varphi\in \LtwozerobdOmega$. In other words, \eqnref{eqn:KstarOMonKmonehalf} and \eqnref{eqn:SinglelayeronKmonehalf} are natural extensions of the NP operator and the single-layer potential on $\LtwozerobdOmega$ to $K^{-1/2}_0$.

We observe that
\begin{equation*}
  \big\lVert\KstarOmega[\varphi]\big\rVert_{-1/2}^2=-\int_{-\infty}^\infty (\mathbb{K}\hvarphi)^T \mathbb{S}\overline{(\mathbb{K}\hvarphi)}\, dk\leq \frac{1}{4}\lVert\varphi\rVert^2_{-1/2}\quad\mbox{for any }\varphi\in K^{-1/2}_0.
\end{equation*}
Hence, $\KstarOmega$ is a bounded linear operator on $K^{-1/2}_0$ and its operator norm is bounded by $\frac{1}{2}$. 
Since $\mathbb{S}$ and $\mathbb{K}$ are real diagonal matrices, we have
$\mathbb{S}\overline{\mathbb{K}} = \mathbb{K}^T\mathbb{S}$.
This induces that
\begin{equation*}
  \langle \psi,\KstarOmega[\varphi]\rangle_{-1/2} = \langle \KstarOmega[\psi], \varphi \rangle_{-1/2}\quad\mbox{for all }\varphi,\psi\in K^{-1/2}_0.
\end{equation*}
In other words, $\KstarOmega$ is self-adjoint on $K^{-1/2}_0$.
In view of \eqnref{eqn:KstarOMonKmonehalf}, $\KstarOmega$ is identical to the matrix $\mathbb{K}$ via the transformation $U^{-1}P$. From this, one can infer that the spectrum of $\KstarOmega$ on $K^{-1/2}_0$ lies in the interval $[-1/2,1/2]$, that is the spectrum of $\mathbb{K}$. In Subsection \ref{subsection:spectraltheorem}, we will prove it by deriving the spectral resolution of the NP operator on $K^{-1/2}_0$.

\smallskip

In the remainder of this subsection, we obtain properties of the layer potential operators by assuming a decay condition on the density function as $k$ tends to zero.
\begin{lemma}\label{Scal:behavior}
Let $\varphi\in  K^{-1/2}_0$ be given by $\varphi=U^{-1}P\hvarphi$ satisfying $\widehat{\varphi}_1,\widehat{\varphi}_2\in L^1(\RR)$ and $\widehat{\varphi}_2(k)=O(|k|^{\frac{1}{4}})$ as $|k|\rightarrow 0$.
Set $z=z_1+iz_2=\Psi(x+iy)\in\CC$. Then, we have the following. 
\begin{enumerate}[(a), left=0.5em]
\item The single-layer potential $\Scal_{\p\Om}[\varphi](z)$ is continuous and uniformly bounded in $\CC$ and $\SingleOmega[\varphi](z)=O(|z|^{-1})$ as $|z|\rightarrow\infty$. 
\item The partial derivatives $\p_x \Scal_{\p\Om}[\varphi](z), \p_y \Scal_{\p\Om}[\varphi](z)$ are uniformly bounded for $x\neq\frac{1}{2R},\frac{1}{2r}$ and 
\begin{equation*}
  \frac{\p}{\p z_1} \SingleOmega[\varphi](z),\ \frac{\p}{\p z_2} \SingleOmega[\varphi](z)=(x^2+y^2)\, O\big(|x|^{-\frac{5}{8}}\big)\quad\mbox{as }|x|\rightarrow\infty,
\end{equation*}
where $O(|x|^{-\frac{5}{8}})$ is uniform with respect to $y$. 
\item The single-layer potential is harmonic, i.e., $\Delta_{(z_1,z_2)}\Scal_{\p\Om}[\varphi](z)=0$ in $\CC\setminus\p\Om$.
\end{enumerate}
\end{lemma}
\begin{proof}

From the assumption that $\widehat{\varphi}_2(k)=O(|k|^{\frac{1}{4}})$, the integral \eqnref{eqn:SinglelayeronKmonehalf} is finite for any $z\in\CC$. 
One can also show that $\Scal_{\p\Om}[\varphi](z)$ is continuous in the whole complex plane by applying the dominated convergence theorem. We can rewrite \eqnref{eqn:SinglelayeronKmonehalf} as
\beq \label{eqn:SinglelayeronKmonehalf3}
 \begin{aligned}
  \SingleOmega[\varphi](z)= &-\frac{1}{\sqrt{2\pi}}\int_{-\infty}^\infty\frac{1}{2|k|}
  \Big((e^{-|x-\frac{1}{2R}||k|} - e^{-\frac{|k|}{2R}})e^{iky}
   + e^{-\frac{|k|}{2R}}(e^{iky}-1)\Big)\hvarphi_R(k)\,dk\\
   &-\frac{1}{\sqrt{2\pi}}\int_{-\infty}^\infty\frac{1}{2|k|}
   \Big((e^{-|x-\frac{1}{2r}||k|} - e^{-\frac{|k|}{2r}})e^{iky}
   + e^{-\frac{|k|}{2r}}(e^{iky}-1)\Big)\hvarphi_r(k)\,dk.
         \end{aligned}
   \eeq
It then follows that  
\begin{equation*}
  \SingleOmega[\varphi](z)=\int_{-\infty}^\infty O(|x|+|y|)|\hvarphi_R(k)|\,dk + \int_{-\infty}^\infty O(|x|+|y|)|\hvarphi_r(k)|\,dk\quad\mbox{as }|x+iy|\rightarrow 0,
\end{equation*}
where $O(|x|+|y|)$ terms are uniform with respect to $k$. 
%Hence, we have $\SingleOmega[\varphi](\Psi^{-1}(x+iy))=O(|x+iy|)$ as $|x+iy|\rightarrow 0$. 
%Hence, we have
%$\SingleOmega[\varphi](z)=O(|z|^{-1})$ as $|z|\rightarrow\infty$. 
This proves (a). 

From \eqnref{eqn:SinglelayeronKmonehalf}, we have
\beq
\begin{aligned}
\frac{\p}{\p y} \SingleOmega[\varphi](z)=&-\frac{1}{\sqrt{2\pi}}\int_{-\infty}^\infty\frac{i k}{2|k|}
\sqrt{2} e^{-|x-\frac{1}{2R}||k|} \hvarphi_2(k)e^{iky}\,dk\\
&-\frac{1}{\sqrt{2\pi}}\int_{-\infty}^\infty\frac{i k}{2|k|}
\Big(e^{-|x-\frac{1}{2r}||k|} - e^{-|x-\frac{1}{2R}||k|}\Big)\hvarphi_r(k)e^{i k y}\,dk
\end{aligned}
\eeq
and
\beq\notag
\frac{\p}{\p x} \SingleOmega[\varphi](z)=
\begin{dcases}
\frac{1}{2\sqrt{2\pi}}\int_{-\infty}^\infty
\Big(e^{-(x-\frac{1}{2R})|k|}\,\hvarphi_R(k)+e^{-(x-\frac{1}{2r})|k|}\,\hvarphi_r(k)\Big)e^{iky}\,dk\quad&\mbox{for }x>\frac{1}{2r},\\[1mm]
\frac{1}{2\sqrt{2\pi}}\int_{-\infty}^\infty
\Big(e^{-(x-\frac{1}{2R})|k|}\,\hvarphi_R(k)-e^{(x-\frac{1}{2r})|k|}\,\hvarphi_r(k)\Big)e^{iky}\,dk\quad&\mbox{for }\frac{1}{2R}<x<\frac{1}{2r},\\[1mm]
\frac{1}{2\sqrt{2\pi}}\int_{-\infty}^\infty
\Big(-e^{(x-\frac{1}{2R})|k|}\,\hvarphi_R(k)-e^{(x-\frac{1}{2r})|k|}\,\hvarphi_r(k)\Big)e^{iky}\,dk\quad&\mbox{for }x<\frac{1}{2R}.
\end{dcases}
\eeq
Because of $\hvarphi_1,\hvarphi_2\in L^1(\RR)$, $\frac{\p}{\p x} \SingleOmega[\varphi](z)$ and $\frac{\p}{\p y} \SingleOmega[\varphi](z)$ are uniformly bounded in $\CC$. 
We have
\begin{align*}
\frac{\p}{\p x} \SingleOmega[\varphi](z),\ \frac{\p}{\p y} \SingleOmega[\varphi](z)
&=\int_{-\infty}^\infty e^{-|x-\frac{1}{2R}||k|}O\big(|k|^{\frac{1}{4}}\big)\big(1+\hvarphi_r(k)\big)dk\\
&=O\big(|x|^{-\frac{5}{8}}\big)\quad\mbox{as }|x|\rightarrow\infty,
\end{align*}
where the second equality can be derived by splitting the integral into $|k|<|x-\frac{1}{2R}|^{-\frac 1 2}$ and $|k|>|x-\frac{1}{2R}|^{-\frac 1 2}$.
This proves (b).

Recall that $\Psi$ is a conformal mapping. 
By taking the Laplacian for the right-hand side of \eqnref{eqn:SinglelayeronKmonehalf} (switching the order of differentiation and integration), we observe (c).
\end{proof}

\begin{lemma}\label{lemma:gradientvalue}
Let $\varphi\in  K^{-1/2}_0$ be given by $\varphi=U^{-1}P\hvarphi$ satisfying $\widehat{\varphi}_1,\widehat{\varphi}_2\in L^1(\RR)$ and $\widehat{\varphi}_2(k)=O(|k|^{\frac{1}{4}})$ as $|k|\rightarrow 0$. We have
\begin{align}\label{gradientvalue:int1}
  &\big\lVert\nabla \SingleOmega[\varphi]\big\rVert_{L^2(\Omega)}^2 
 = \Big\langle \varphi,\Big(\frac{1}{2}I-\KstarOmega\Big)[\varphi] \Big\rangle_{-1/2}<\infty,\\
\label{gradientvalue:int2}
  &\big\lVert\nabla \SingleOmega[\varphi]\big\rVert_{L^2(\RR^2\setminus \overline{\Omega})}^2 
 = \Big\langle \varphi,\Big(\frac{1}{2}I+\KstarOmega\Big)[\varphi] \Big\rangle_{-1/2}<\infty.
\end{align}
\end{lemma} 
\begin{proof}
We define $\hvarphi_R$ and $\hvarphi_r$ as in \eqnref{def:hvarphi_R_r}. Note that 
\beq\label{varphi:deday}
\widehat{\varphi}_R(k)+\widehat{\varphi}_r(k)=\sqrt{2}\, \widehat{\varphi}_2(k)=O(|k|^{\frac{1}{4}})\quad\text{as } |k|\rightarrow 0.
\eeq
For fixed $s>0$, we set 
\begin{equation*}
  \Om_s := \bigg\{z=\Psi(w)\,\Big\vert\, w = x+iy\mbox{ with }\frac{1}{2R}<|x|<\frac{1}{2r},~ |y|<s\bigg\}.
\end{equation*}
Then, $\Om_s$ is a Lipschitz domain. 
We identify $z=z_1+i z_2$ with $\bfz=(z_1,z_2)\in\RR^2$. 
%Note that $\Omega_s:=\Psi(S_s)$ is a Lipschitz domain, where $S_s$ is given by 
%$S_w=\big\{(x,y)\in\mathbb{R}^2\,\big|\,\frac{1}{2R}<|x|<\frac{1}{2r},~ |y|<s\big\}$.
 Applying the divergence theorem, we obtain
\begin{align}
  \int_{\Omega_s}|\nabla\SingleOmega[\varphi]|^2\, d\bfz \notag
  =&\int_{\partial\Omega_s}\overline{\SingleOmega[\varphi]}\, \frac{\partial}{\partial\nu}\SingleOmega[\varphi]\Big|^- d\sigma \notag\\
  =&\int_{x=\frac{1}{2R},\,|y|<s} \overline{\SingleOmega[\varphi]}\,\frac{\partial}{\partial\nu}\SingleOmega[\varphi]\Big|^-h_R(y) dy
  +\int_{x=\frac{1}{2r},\,|y|<s}\overline{\SingleOmega[\varphi]}\,\frac{\partial}{\partial\nu}\SingleOmega[\varphi]\Big|^-h_r(y) dy\notag\\\notag
  &\ +\int_{\frac{1}{2R}<x<\frac{1}{2r},\,|y|=s}\overline{\SingleOmega[\varphi]}\,\frac{\partial}{\partial\nu}\SingleOmega[\varphi]\Big|^- h(x,y) dx\\\label{eqn:I_II_III}
  =&:I+II+III.
\end{align}

We first estimate $III$ as $s\rightarrow\infty$. For $(x,y)$ in the domain of the integral $III$, we have
\begin{align}
		&\frac{\partial}{\partial\nu}\SingleOmega[\varphi]\Big|^- h(x,y)
		=\pm\frac{\partial}{\partial y}\SingleOmega[\varphi]\Big|_{y=\pm s} \notag \\
		=&\mp \frac{1}{\sqrt{2\pi}}\bigg(\int_{-\infty}^\infty\frac{ik}{2|k|}e^{-|x-\frac{1}{2R}||k|}\, \widehat{\varphi}_R(k)e^{\pm iks}\, dk
		+\int_{-\infty}^\infty\frac{ik}{2|k|}e^{-|x-\frac{1}{2r}||k|}\,\widehat{\varphi}_r(k)e^{\pm iks}\, dk
		\bigg),\label{Scalderiv:III}
\end{align}
where the first equality holds similarly to \eqnref{normal_deriv} and the second one is from \eqnref{eqn:SinglelayeronKmonehalf}.
Applying the Riemann--Lebesgue lemma to \eqnref{Scalderiv:III}, we obtain
\begin{align*}
 &\frac{\partial}{\partial\nu}\SingleOmega[\varphi]\Big|^- h(x,y)
 \rightarrow 0\quad\mbox{as }s\rightarrow\infty.
\end{align*}
Note that $\SingleOmega[\varphi](z)$ and $h(x,y)\nabla_{(z_1,z_2)} \SingleOmega[\varphi](z)$ are uniformly bounded independent of $s$ for $(x,y)$ satisfying $\frac{1}{2R}<x<\frac{1}{2r},~|y|\geq s$. It then holds by applying the dominated convergence theorem that
\beq\label{eqn:III}
III\rightarrow 0\quad\mbox{as }s\rightarrow\infty.
\eeq

Now, we estimate $I+II$. 
From \eqnref{normal_deriv} and \eqnref{eqn:SinglelayeronKmonehalf}, we have 
\begin{align*}
\SingleOmega[\varphi](z)
   =&-{\mathcal{F}}^{-1}\Big[\frac{1}{2|k|}
   \big({\hvarphi_R(k)}  +e^{-|k|q}\,{\hvarphi_r(k)}\big)\Big](y) +C\quad\mbox{in }I,\\
   %%%%
  \SingleOmega[\varphi](z)
%   = &-\frac{1}{\sqrt{2\pi}}\int_{-\infty}^\infty\frac{1}{2|k|}
%   (e^{-|k|q}\,\hvarphi_R(k)
%   +\hvarphi_r(k)  ) e^{iky}\, dk +C\quad\mbox{in }II\\
   = &-{\mathcal{F}}^{-1}\Big[\frac{1}{2|k|}
   \big(e^{-|k|q}\,{\hvarphi_R(k)}
   +{\hvarphi_r(k)}\big)\Big] (y)+{C}\quad\mbox{in }II,
      \end{align*}
where $C$ is the constant given by $C=\frac{1}{\sqrt{2\pi}}\int_{-\infty}^\infty\frac{1}{2|k|}
   \big(e^{-\frac{|k|}{2R}}\,\hvarphi_R(k)+e^{-\frac{|k|}{2r}}\,\hvarphi_r(k)\big) dk$.
   We also have 
\begin{align*}
 \frac{\partial}{\partial\nu}\SingleOmega[\varphi]\Big|^-  h_R(y)
  &=-\frac{\partial}{\partial x}\SingleOmega[\varphi]\Big|^+_{x=\frac{1}{2R}} % \label{Scalderiv:I}
    = \frac{1}{2}\,{\mathcal{F}}^{-1}\big[-\widehat{\varphi}_R(k)+ e^{-|k|q}\,\widehat{\varphi}_r(k)\big](y)\quad\mbox{in }I,\\[1mm]  %\label{Scalderiv:II}
 \frac{\partial}{\partial\nu}\SingleOmega[\varphi]\Big|^-  h_r(y)
 &=\frac{\partial}{\partial x}\SingleOmega[\varphi]\Big|^-_{x=\frac{1}{2r}}
    =-\frac{1}{2}\,{\mathcal{F}}^{-1}\big[- e^{-|k|q}\,\widehat{\varphi}_R(k)
   +\widehat{\varphi}_r(k)\big](y)\quad\mbox{in }II.
\end{align*}
In other words,
\begin{align}\label{Scal:fourier}
\begin{bmatrix}
\mathcal{F}\big[{\SingleOmega[\varphi]}\big|_{x=\frac{1}{2R}}-C\big] \\
\mathcal{F}\big[{\SingleOmega[\varphi]}\big|_{x=\frac{1}{2r}}-C\big]
\end{bmatrix}^T
=\begin{bmatrix}
     \widehat{\varphi}_R\\\widehat{\varphi}_r
   \end{bmatrix}^T P^{-1}\,\mathbb{S}\,P
\end{align}
and
\begin{align}\label{Scal:normal:fourier1}
 \begin{bmatrix}
 \mathcal{F}\big[\frac{\partial}{\partial\nu}\SingleOmega[\varphi]\Big|^-  h_R\big]\\[2mm]
  \mathcal{F}\big[\frac{\partial}{\partial\nu}\SingleOmega[\varphi]\Big|^-  h_r\big]
 \end{bmatrix}
 = -P^{-1}(\frac{1}{2}I-\mathbb{K})P
 \begin{bmatrix}
     \widehat{\varphi}_R\\\widehat{\varphi}_r
   \end{bmatrix}.
 \end{align}
Applying the Plancherel theorem, we derive that as $s\rightarrow\infty$,
\begin{align}\notag
I+II\,\rightarrow\,
&\int_{-\infty}^\infty \overline{\begin{bmatrix}
 \widehat{\varphi}_1\\\widehat{\varphi}_2 \end{bmatrix}}^T(-\mathbb{S})(\frac{1}{2}I-\mathbb{K})
 {\begin{bmatrix}
     \widehat{\varphi}_1\\\widehat{\varphi}_2
   \end{bmatrix}}dk+C'
   \end{align}
with
 \begin{align*}
 C' =&\frac{\overline{C}}{2}\lim_{s\rightarrow\infty}\int_{-s}^s
     \bigg[\mathcal{F}^{-1}(\widehat{\varphi}_R)+\mathcal{F}^{-1}(\widehat{\varphi}_r)
   -\mathcal{F}^{-1}\big(e^{-|k|q}\,\widehat{\varphi}_R\big)
   -\mathcal{F}^{-1}\big(e^{-|k|q}\,\widehat{\varphi}_r\big)\bigg]
   dy.
%   &\qquad\times\int_{-\infty}^\infty\frac{1}{2|k|}\bigg( e^{-\frac{|k|}{2R}}\, \overline{\widehat{\varphi}_R}+e^{-\frac{|k|}{2r}}\, \overline{\widehat{\varphi}_r}\bigg)dk.
\end{align*}
We claim that $C'=0$. 
Let $\widehat{g}\in L^1(\mathbb{R})$ and $\widehat{g}(k)=O(|k|^{\frac{1}{4}})$ as $|k|\rightarrow 0$. It then holds by Fubini's theorem and the dominated convergence theorem that
\begin{align*}
\int_{-s}^s \mathcal{F}^{-1}(\, \widehat{g}\,)dy
 &=\int_{-s}^s \int_{-\infty}^\infty\widehat{g}(k)e^{iky}\,dk\,dy
 =\int_{-\infty}^\infty \widehat{g}(k)\int_{-s}^se^{iky}\,dy \, dk
 =2\int_{-\infty}^\infty \frac{\widehat{g}(k)}{k} \,\sin(ks)\,dk.
\end{align*}
By the Riemann--Lebesgue lemma, the last term converges to $0$ as $s\rightarrow \infty$ and, thus,\begin{equation*}
 \lim_{s\rightarrow\infty}\int_{-s}^s \mathcal{F}^{-1}(\, \widehat{g}\, )dy=0.
\end{equation*}
This implies that $C'=0$. From \eqnref{def:single}, \eqnref{eqn:KstarOMonKmonehalf}, \eqnref{eqn:I_II_III} and \eqnref{eqn:III}, we prove \eqnref{gradientvalue:int1}.

\smallskip

Note that 
\begin{equation*}
  \Om^e = \bigg\{z=\Psi(w) \,\Big\vert\, w = x+iy\mbox{ with }x\in(\frac{1}{2r},\infty)\cup(-\infty,\frac{1}{2R})\bigg\}.
\end{equation*}
From Lemma \ref{Scal:behavior}\,(b), we obtain
\begin{equation*}
  \int_{\Omega^e}|\nabla\SingleOmega[\varphi]|^2\leq C\int_{x\in (-\infty,\frac{1}{2R})\cup (\frac{1}{2r},\infty)}\int_{-\infty}^\infty \, |x|^{-\frac{5}{8}} \frac{1}{x^2+y^2}\, dy \,dx<\infty
\end{equation*}
and 
\begin{align}
 &\int_{\Omega^e}|\nabla\SingleOmega[\varphi]|^2\, d\bfz \notag\\ \notag
  =&-\int_{\partial\Omega^e}\overline{\SingleOmega[\varphi]}\, \frac{\partial}{\partial\nu}\SingleOmega[\varphi]\Big|^+ d\sigma \notag\\
  =&-\int_{x=\frac{1}{2R},\,y\in\RR} \overline{\SingleOmega[\varphi]}\,\frac{\partial}{\partial\nu}\SingleOmega[\varphi]\Big|^+ h_R(y) dy
  -\int_{x=\frac{1}{2r},\,y\in\RR}\overline{\SingleOmega[\varphi]}\,\frac{\partial}{\partial\nu}\SingleOmega[\varphi]\Big|^+ h_r(y) dy\notag\\\notag
 \\\notag
  =&:I+II.
\end{align}
%As in the discussion above, we have $III\rightarrow0$ as $s\rightarrow\infty$. 
Note that
\begin{align*}
 \frac{\partial}{\partial\nu}\SingleOmega[\varphi]\Big|^+  h_R(y)
  &=-\frac{\partial}{\partial x}\SingleOmega[\varphi]\Big|^-_{x=\frac{1}{2R}}  
    = \frac{1}{2}\,{\mathcal{F}}^{-1}\big[\widehat{\varphi}_R(k)+ e^{-|k|q}\,\widehat{\varphi}_r(k)\big](y)\quad\mbox{in }I,\\[1mm]
 \frac{\partial}{\partial\nu}\SingleOmega[\varphi]\Big|^+  h_r(y)
 &=\frac{\partial}{\partial x}\SingleOmega[\varphi]\Big|^+_{x=\frac{1}{2r}}
    =-\frac{1}{2}\,{\mathcal{F}}^{-1}\big[- e^{-|k|q}\,\widehat{\varphi}_R(k)
   -\widehat{\varphi}_r(k)\big](y)\quad\mbox{in }II,
   \end{align*}
  which implies that
  \begin{align}\label{Scal:normal:fourier2}
 \begin{bmatrix}
 \mathcal{F}\Big[\frac{\partial}{\partial\nu}\SingleOmega[\varphi]\Big|^+  h_R\Big]\\[2mm]
  \mathcal{F}\Big[\frac{\partial}{\partial\nu}\SingleOmega[\varphi]\Big|^+  h_r\Big]
 \end{bmatrix}
 = P^{-1}(\frac{1}{2}I+\mathbb{K})P
 \begin{bmatrix}
     \widehat{\varphi}_R\\\widehat{\varphi}_r
   \end{bmatrix}.
 \end{align}
Applying also \eqnref{Scal:fourier}, one can derive that
\begin{align}\notag
I+II= \mbox{const.}+\int_{-\infty}^\infty \overline{\begin{bmatrix}
 \widehat{\varphi}_1\\\widehat{\varphi}_2 \end{bmatrix}}^T(-\mathbb{S})(\frac{1}{2}I+\mathbb{K})
 {\begin{bmatrix}
     \widehat{\varphi}_1\\\widehat{\varphi}_2
   \end{bmatrix}}dk.
   \end{align}
One can show that the constant term is zero similarly to the proof of $C'=0$. 
Hence, we prove \eqnref{gradientvalue:int2}.

\end{proof}

\subsection{Spectral resolution of $\KstarOmega$ on $K^{-1/2}_0$}\label{subsection:spectraltheorem}
To derive the spectral resolution of the NP operator on $K^{-1/2}_0$, we define a pair of orthogonal projection operators $\mathcal{P}_1(s)$ and $\mathcal{P}_2(s)$ on $K^{-1/2}_0$ for each $s\in\RR\cup\{\infty\}$.
Let $\varphi\in K^{-1/2}_0$ be given by $\varphi=U^{-1}P\hvarphi$ and $\hvarphi=[\hvarphi_1,\hvarphi_2]^T$. We define 
\begin{align*}
\mathcal{P}_1(s)\varphi&= U^{-1}P\begin{bmatrix}
      \chi_{(-\infty,\, s\,]}\, \widehat{\varphi}_1\\[2mm]
        0
      \end{bmatrix},\quad
\mathcal{P}_2(s)\varphi = U^{-1}P\begin{bmatrix}
       0\\
     \chi_{(-\infty,\, s\,]}\, \widehat{\varphi}_2
      \end{bmatrix}
            \end{align*}
for $s\in\RR$ and 
\begin{equation*}
  \mathcal{P}_1(\infty)\varphi=U^{-1}P\begin{bmatrix}
        \widehat{\varphi}_1\\[2mm]
          0
        \end{bmatrix}
        ,\quad
        \mathcal{P}_2(\infty)\varphi = U^{-1}P\begin{bmatrix}
         0\\
       \widehat{\varphi}_2
        \end{bmatrix}.
\end{equation*}
 Note that $\mathcal{P}_1(\infty)+\mathcal{P}_2(\infty)=\mathbb{I},$ where $\mathbb{I}$ is the identity operator on $K^{-1/2}_0$.
We then define a family of projection operators $\mathbb{E}(t)$ on $K^{-1/2}_0$, $t\in [-1/2,1/2]$, by
\begin{equation}\label{def:resolutionoftheidentity}
  \mathbb{E}(t)
  :=\begin{cases}
    \mathcal{P}_1\Big(-\frac{\ln(-2t)}{q}\Big)-\mathcal{P}_1\Big(\frac{\ln(-2t)}{q} \Big), \quad&t\in[-1/2,0),\\[2mm]
   \mathcal{P}_1(\infty),\quad&t=0,\\[2mm]
    -\mathcal{P}_2\Big(-\frac{\ln(2t)}{q} \Big)+\mathcal{P}_2\Big(\frac{\ln(2t)}{q}\Big)+\mathbb{I}, \quad&t\in(0, 1/2].
   \end{cases}
\end{equation}
It is straightforward to obtain that, for $s,t\in[-1/2,1/2]$,
\beq\label{resolution:I}
\begin{cases}
    \mathbb{E}(t)\,\mathbb{E}(s)=\mathbb{E}(\min(t,s)),\\
    \lim_{t\rightarrow s^+}\mathbb{E}(t)=\mathbb{E}(s),\\
    \mathbb{E}(-\frac{1}{2})=0,\text{ and } \mathbb{E}(\frac{1}{2}) =\mathbb{I}.
\end{cases}
\eeq
The limit in \eqnref{resolution:I} is in the sense of strong convergence, i.e., 
$\lim_{t\rightarrow s^+}\mathbb{E}(t)\psi=\mathbb{E}(s)\psi$ for all $\psi\in K^{-1/2}_0$,
which holds from the dominated convergence theorem and the fact that the integral in \eqnref{def:K:norm} is finite. In short, we have the following lemma.
\begin{lemma}\label{prop:innerproductrep}
The family of operators $\{\mathbb{E}(t)\}_{t\in[-\frac 1 2, \frac 1 2]}$ is a resolution of identity on $K^{-1/2}_0$ and satisfies
 \beq
 \lim_{t\rightarrow s}\mathbb{E}(t)=\mathbb{E}(s)
 \eeq
 in the sense of strong convergence. 
 \end{lemma}

We have
\begin{equation*}
  \la \psi,\mathcal{P}_1(s)\varphi\ra_{-1/2}= \int_{-\infty}^s
     \begin{bmatrix}
          \widehat{\psi}_1\\
          \widehat{\psi}_2
        \end{bmatrix}^T
  (-\mathbb{S})~
     \overline{\begin{bmatrix}
           \widehat{\varphi}_1\\
          0
         \end{bmatrix}}\, dk,
\end{equation*}
 which is almost everywhere differentiable in $s$ from the Lebesgue differentiation theorem. Similarly, $\big\langle \psi,\, \mathbb{E}(t)\varphi\big\rangle_{-1/2}$ is almost everywhere differentiable in $t$. 
\begin{lemma}\label{lemma:resolution}
 For all $\varphi,\psi\in K^{-1/2}_0$, it holds that
  \begin{equation}\label{psi:E:phi}
    \langle \psi,\varphi\rangle_{-1/2}=\int_{-\frac{1}{2}}^{\frac 1 2}d\big\langle \psi,\, \mathbb{E}(t)\varphi\big\rangle_{-1/2}.
  \end{equation}
\end{lemma}
\begin{proof}
Let $\varphi,\psi$ be given by \eqnref{varphi:hat}.
It then holds from \eqnref{def:single} that 
\begin{align}\label{eqn:innerpossms}
 &\langle \psi, \,\mathcal{P}_1(\pm k)\varphi\rangle_{-1/2} 
  =\frac{1}{2}\int_{-\infty}^{\pm k} \widehat{\psi}_1(s)\, \overline{\widehat \varphi_1(s)}\, \frac{1-e^{-q|s|}}{|s|}\, ds\quad\mbox{for }k>0.
%\label{eqn:innerpossps}
%&\langle \psi, \,\mathcal{P}_2(\pm \xi)\varphi\rangle_{-1/2} 
%  =\frac{1}{2}\int_{-\infty}^{\pm \xi} \widehat{\psi}_2(s)\, \overline{\widehat \varphi_2(s)}\, \frac{1+e^{-q|s|}}{|s|}\, ds.
\end{align}
From \eqnref{eqn:innerpossms} and a similar derivation, we obtain
\begin{align*}
\frac{d}{dk}\la \psi,\big(\mathcal{P}_1(k)-\mathcal{P}_1(-k)\big)\varphi\ra_{-1/2}&=\frac{1}{2}\,\widehat \psi_1(k) \, \overline{\widehat \varphi_1(k)}\, \frac{1-e^{-q|k|}}{|k|} + \frac{1}{2}\,\widehat \psi_1(-k) \, \overline{\widehat \varphi_1(-k)}\, \frac{1-e^{-q|k|}}{|k|},\\
%\end{align*}
%and, by a similar computation,
%\begin{align*}
\frac{d}{dk}\la \psi,\big(\mathcal{P}_2(k)-\mathcal{P}_2(-k)\big)\varphi\ra_{-1/2}&=\frac{1}{2}\,\widehat \psi_2(k)\, \overline{\widehat \varphi_2(k)}\, \frac{1+e^{-q|k|}}{|k|} + \frac{1}{2}\,\widehat \psi_2(-k) \, \overline{\widehat \varphi_2(-k)}\, \frac{1+e^{-q|k|}}{|k|}.
\end{align*}
We then obtain that, letting $t =-\frac{1}{2}e^{-qk}\in[-1/2, 0)$,
\begin{align*}
  \frac{1}{2}\int_{-\infty}^{\infty}\widehat \psi_1(k) \, \overline{\widehat \varphi_1(k)}\, \frac{1-e^{-q|k|}}{|k|}\, dk
  =&\int_0^\infty \frac{d}{dk}\la \psi,\big(\mathcal{P}_1(k)-\mathcal{P}_1(-k)\big)\varphi\ra_{-1/2}dk\\
  =&\int_{-\frac 1 2}^{0}  d\Big\langle \psi,\, \bigg(\mathcal{P}_1\Big(-\frac{\ln(-2t)}{q}\Big)-\mathcal{P}_1\Big(\frac{\ln(-2t)}{q} \Big)\bigg)\,  \varphi\Big\rangle_{-1/2}
\end{align*}
and that, letting $t =\frac{1}{2}e^{-qk}\in (0,1/2]$,
\begin{align*}
  \frac{1}{2}\int_{-\infty}^{\infty}\widehat \psi_2(k) \, \overline{\widehat \varphi_2(k)}\, \frac{1+e^{-q|k|}}{|k|}\, dk
  =&\int_0^\infty \frac{d}{dk}\la \psi,\big(\mathcal{P}_2(k)-\mathcal{P}_2(-k)\big)\varphi\ra_{-1/2}dk\\
  =&\int_{\frac 1 2}^0  d\Big\langle \psi,\bigg(\mathcal{P}_2\Big(-\frac{\ln(2t)}{q}\Big)-\mathcal{P}_2\Big(\frac{\ln(2t)}{q} \Big)\bigg) \varphi\Big\rangle_{-1/2}.
\end{align*}
Since $\la \psi,\,\mathbb{I}\varphi\ra_{-1/2}$ is constant, we have $d\la \psi,\,\mathbb{I}\varphi\ra_{-1/2}=0$. Hence, we complete the proof.
\end{proof}

\begin{theorem}\label{thm:spectrumtheoremKstarOmega}
  Let $\big\{\mathbb{E}(t)\big\}_{t\in[-1/2,1/2]}$ be the resolution of the identity on $K^{-1/2}_0$ given by \eqnref{def:resolutionoftheidentity}. We have the spectral resolution of $\KstarOmega$ on $K^{-1/2}_0$ as
  \begin{equation*}
    \KstarOmega = \int_{-\frac 1 2}^{\frac 1 2}\,t\, d\,\mathbb{E}(t).
  \end{equation*}
  In other words, it holds that
  \begin{equation}\label{eqn:spectrumtheoremKstarOmega}
    \big\langle \psi, \KstarOmega[\varphi] \big\rangle_{-1/2}= \int_{-\frac 1 2}^{\frac 1 2}\,t\, d\,\big\langle \psi, \mathbb{E}(t)\varphi \big\rangle_{-1/2}\quad\mbox{for all }\psi,\varphi\in K^{-1/2}_0.
  \end{equation}
\end{theorem}
\begin{proof}
Let $\varphi,\psi$ be given by \eqnref{varphi:hat}. From the definition of the NP operator and the inner product on $K^{-1/2}_0$, we have \begin{align*}
    \langle \psi, \KstarOmega[\varphi] \rangle_{-1/2}
     =&-\frac{1}{4}\int_{-\infty}^\infty e^{-q|k|}\, \widehat{\psi}_1\, 
     \overline{\widehat{\varphi}_1}\, \frac{1-e^{-q|k|}}{|k|}\, dk
     +\frac{1}{4}\int_{-\infty}^\infty e^{-q|k|}\, \widehat{\psi}_2\, 
     \overline{\widehat{\varphi}_2} \, \frac{1+e^{-q|k|}}{|k|}\, dk\\
     =&:I+II.
     \end{align*}
Following the derivation in the proof of Lemma \ref{lemma:resolution}, we obtain
\begin{align*}
I=&\int_0^\infty \Big(-\frac 1 2 e^{-q|k|}\Big)\frac{d}{dk}\la \psi,\big(\mathcal{P}_1(k)-\mathcal{P}_1(-k)\big)\varphi\ra_{-1/2}dk\\
  =&\int_{-\frac 1 2}^{0} t\, d\Big\langle \psi,\, \bigg(\mathcal{P}_1\Big(-\frac{\ln(-2t)}{q}\Big)-\mathcal{P}_1\Big(\frac{\ln(-2t)}{q} \Big)\bigg)\,  \varphi\Big\rangle_{-1/2}
  \end{align*}
  and 
  \begin{align*}
  II =&\int_0^\infty \Big( \frac 1 2 e^{-q|k|} \Big)\frac{d}{dk}\la \psi,\big(\mathcal{P}_2(k)-\mathcal{P}_2(-k)\big)\varphi\ra_{-1/2}dk \\
=&\int_{\frac 1 2}^0 t \, d\Big\langle \psi,\bigg(\mathcal{P}_2\Big(-\frac{\ln(2t)}{q}\Big)-\mathcal{P}_2\Big(\frac{\ln(2t)}{q} \Big)\bigg) \varphi\Big\rangle_{-1/2}.
  \end{align*}
  This proves the theorem. 
\end{proof}
The condition
$
  \lim_{t\rightarrow s^-}\mathbb{E}(t) \neq \mathbb{E}(s)
$
characterizes the point spectrum of $\KstarOmega$.
By Lemma \ref{prop:innerproductrep}, we conclude that $\KstarOmega$ has only a continuous spectrum. We need the following lemma to prove $\KstarOmega$ has only an absolutely continuous spectrum.

\begin{lemma}\label{lemma:measuredensity}
We have
\begin{align*}
  \frac{d}{dt}\big\langle \psi,\, \mathbb{E}(t)\varphi\big\rangle_{-1/2}
  =\begin{dcases}
  \frac{1+2t}{2|t\ln(2|t|)|}\Big(\,\hpsi_1(k)\, \overline{ \hvarphi_1(k)}+\hpsi_1(-k)\, \overline{\hvarphi_1(-k)}\, \Big) ,\quad&t\in[-1/2,0),\\[1mm]
  \frac{1+2t}{2|t\ln(2|t|)|}\Big(\,\hpsi_2(k)\, \overline{ \hvarphi_2(k)}+\hpsi_2(-k)\, \overline{\hvarphi_2(-k)}\, \Big), \quad&t\in(0,1/2]
  \end{dcases}
\end{align*}
with  $k = -\frac{1}{q}\ln(2|t|)$.
\end{lemma}
\begin{proof}
First, we consider the case $t\in[-1/2,0)$. From the derivation in the proof of Lemma \ref{lemma:resolution}, we have $\mathbb{E}(t)=P_1(k)-P_1(-k)$ with $t=-\frac{1}{2}e^{-qk}$ ($k>0$)
and 
\begin{align}\notag
\la \psi,\mathbb{E}(t)\varphi\ra_{-1/2}&=\frac{1}{2}\int_{-k}^k \hpsi_1(s)\,\overline{\hvarphi_1(s)}\,\frac{1-e^{-q|s|}}{|s|}\,ds,\\ \label{eqn1:lemma:measuredensity}
 \frac{d}{dt}\big\langle \psi,\, \mathbb{E}(t)\varphi\big\rangle_{-1/2}
 &=\frac{1}{2}\Big(\,\hpsi_1(k)\, \overline{ \hvarphi_1(k)}+\hpsi_1(-k)\, \overline{\hvarphi_1(-k)}\, \Big)\frac{1-e^{-q|k|}}{|k|}\frac{dk}{dt}.
 \end{align}
 
Now, let $t\in(0,1/2]$. We have $\mathbb{E}(t)=-P_2(k)+P_2(-k)+\mathbb{I}$ with $t=\frac 1 2 e^{-qk}$ ($k>0$) and 
\begin{align}\notag
\la \psi,\mathbb{E}(t)\varphi\ra_{-1/2}&=-\frac{1}{2}\int_{-k}^k \hpsi_2(s)\,\overline{\hvarphi_2(s)}\,\frac{1+e^{-q|s|}}{|s|}\,ds+\mbox{const.},\\ \label{eqn2:lemma:measuredensity}
 \frac{d}{dt}\big\langle \psi,\, \mathbb{E}(t)\varphi\big\rangle_{-1/2}
 &=-\frac{1}{2}\Big(\,\hpsi_2(k)\, \overline{ \hvarphi_2(k)}+\hpsi_2(-k)\, \overline{\hvarphi_2(-k)}\, \Big)\frac{1+e^{-q|k|}}{|k|}\frac{dk}{dt}.
 \end{align}
From \eqnref{eqn1:lemma:measuredensity} and \eqnref{eqn2:lemma:measuredensity}, we prove the lemma.
\end{proof}

For any $\varphi\in K^{-1/2}_0$, it holds from \eqnref{eqn1:lemma:measuredensity} and \eqnref{eqn2:lemma:measuredensity} that
\begin{align}\label{Q:L1}
  \int_{-\frac 1 2}^{\frac 1 2}\bigg|\frac{d}{dt}\big\langle \varphi,\, \mathbb{E}(t)\varphi\big\rangle_{-1/2}\bigg|\, dt 
 &= \frac{1}{2}\int_{-\infty}^\infty\bigg(\frac{1-e^{-q|k|}}{|k|}\, |\hvarphi_1(k)|^2+\frac{1+e^{-q|k|}}{|k|}\, |\hvarphi_2(k)|^2\bigg)\, dk<\infty.
\end{align}
In other words, $\frac{d}{dt}\big\langle \varphi,\, \mathbb{E}(t)\varphi\big\rangle_{-1/2}$ is integrable for any $\psi\in K^{-1/2}_0$, which implies that $\frac{d}{dt}\big\langle \psi,\, \mathbb{E}(t)\varphi\big\rangle_{-1/2}$ is integrable for any $\psi,\varphi\in K^{-1/2}_0$. 
In view of \eqnref{eqn:spectrumtheoremKstarOmega}, we obtain the following theorem.
\begin{theorem}
 The NP operator $\KstarOmega$ on $K^{-1/2}_0$ has only the absolutely continuous spectrum $[-\frac 1 2,\frac 1 2]$.
\end{theorem}

\section{Plasmon resonance on a crescent-shaped domain}\label{subsection:plasmonicresonance}
In this section, we analyze the plasmon resonance on a crescent-shaped domain $\Om=B_R\setminus\overline{B_r}$. 

We consider the transmission problem \eqnref{eqn:u_delta} with $\epsilon$ given by 
\beq\label{ep:Om}
\epsilon = (\epsilon_c+i\delta)\chi_\Om +\chi_{\mathbb{R}^2\setminus\overline{\Om}}.
\eeq
%and $f$ is a source function compactly supported in $\RR^2\setminus\overline{\Om}$ with the zero mean. 
Let $f$ be compactly supported away from $\Om$ so that the Newtonian potential $F$ of $f$ is 
smooth in a neighborhood of $\overline{\Om}$. It then holds from Lemma \ref{lemma:L20:K} that $\p_\nu F\in K_0^{-1/2}$.
%We denote by $F$ the Newtonian potential of $f$. Since $f$ is compactly supported away from $\Om$, $F$ is smooth in a neighborhood of $\p\Om$. Hence, by Lemma \ref{lemma:L20:K}, it holds that $\p_\nu F\in K_0^{-1/2}$.

Set $\lambda_\delta$ as in \eqnref{eqn:intgralequationintro}. Because of $\lambda_\delta\in\mathbb{R}^2\setminus[-\frac 1 2,\frac 1 2]$, the problem
\begin{equation}\label{eqn:intgralequationforplasmonicresonance}
	(\lambda_\delta I -\KstarOmega)\big[\varphi^\delta\big] = \partial_\nu F
	\end{equation}
is solvable in $K^{-1/2}_0$. The solution $\varphi^\delta\in K^{-1/2}_0$ is of the form (see \eqnref{varphi_K})
\begin{equation*}
  \varphi^\delta=U^{-1}P\,\widehat{\varphi}^\delta\quad\mbox{with }\widehat{\varphi}^\delta=\begin{bmatrix}\widehat{\varphi}^{\delta}_{1}\\\widehat{\varphi}^{\delta}_2
     \end{bmatrix},
\end{equation*}
which, by Theorem \ref{thm:spectrumtheoremKstarOmega}, admits the following integral expression:
\begin{align}\label{eqn:varphideltaint}
  \varphi^\delta 
 &= \int_{-\frac 1 2}^{\frac 1 2}\frac{1}{\lambda_\delta-t}\, d\, \mathbb{E}(t)[\partial_{\nu}F].\end{align}

Recall the definition of $U$ in \eqnref{def:U}. For any fixed positive integer $n$, if $g\in L^1(\RR)$ satisfies $\partial^\alpha g\in L^1(\RR)$ and $\partial^\alpha g(x)\rightarrow 0$ as $|x|\rightarrow\infty$ for all $|\alpha|\leq n$, then $
|\mathcal{F}[g](k)|\leq\frac{C}{(1+|k|)^n}
$
with some positive constant $C$.
From the smoothness of $\partial_\nu F$ on $\p B_R$ and $\p B_r$, each component of $U[\partial_\nu F]$ then belongs to $L^1(\mathbb{R})\cap L^2(\mathbb{R})$ so that 
\beq\label{cond:F}
PU[\partial_\nu F]\in L^1(\mathbb{R})\cap L^2(\mathbb{R}).
 \eeq
From \eqnref{def:matrix:S:K}, \eqnref{eqn:KstarOMonKmonehalf} and \eqnref{eqn:intgralequationforplasmonicresonance}, $\hvarphi_\delta$ satisfies
\begin{equation*}
 (\lambda_\delta I -\mathbb{K})[\widehat{\varphi}^\delta] = PU[\partial_\nu F],
\end{equation*}
which leads to
\begin{equation}\label{varphi_delta:matrix}
\widehat{\varphi}^\delta=  \begin{bmatrix}
   ( \lambda_\delta - \frac{1}{2}e^{-|k|q})^{-1} &0\\[2mm]
    0 &(\lambda_\delta + \frac{1}{2}e^{-|k|q})^{-1}
  \end{bmatrix}PU[\partial_\nu F].
\end{equation}
Because of $\p_\nu F\in L^2_0(\p\Om)$, from \eqnref{Fourier:L2_0}, the second entry of $PU[\p_\nu F]$ satisfies $O(|k|^\frac{1}{4})$ as $k\rightarrow 0$. 
From the assumption that $\lambda_\delta\in \RR^2\setminus[-\frac{1}{2},\frac{1}{2}]$, it then holds that 
\begin{equation*}
  \widehat{\varphi}_1^\delta,\widehat{\varphi}_2^\delta\in L^1(\RR)\cap L^2(\RR)\quad\mbox{and}\quad\widehat{\varphi}^{\delta}_{2}\in O(|k|^\frac{1}{4}).
\end{equation*}
% Then, from Lemma \ref{Scal:behavior}\,(a) and Lemma \ref{lemma:gradientvalue}, the single-layer potential $\Scal_{\p\Om}\big[\varphi^\delta\big](z)$ belongs to $H^1(\RR^2)$. As $\varphi^\delta$ satisfies \eqnref{eqn:intgralequationforplasmonicresonance}, the solution to \eqnref{eqn:u_delta} with  $\epsilon$ given by \eqnref{ep:Om} satisfies
% \begin{equation*}
%   u_\delta(z)=F(z)+\Scal_{\p\Om}\big[{\varphi}^\delta\big](z).
% \end{equation*}
Then, from Lemma \ref{Scal:behavior}\,(a), $\Scal_{\p\Om}[\varphi^\delta]$ is continuous in $\RR^2$, harmonic in $\RR^2\setminus\p\Om$, and harmonic at infinity. Furthermore, from Lemma \ref{lemma:gradientvalue}, $|\nabla\Scal_{\p\Om}[\varphi]|\in L^2(\RR^2)$. As $\varphi^\delta$ satisfies \eqnref{eqn:intgralequationforplasmonicresonance}, the solution to \eqnref{eqn:u_delta} with  $\epsilon$ given by \eqnref{ep:Om} satisfies
\begin{equation*}
  u_\delta(z)=F(z)+\Scal_{\p\Om}\big[{\varphi}^\delta\big](z).
\end{equation*}
Since $F$ is smooth in a region containing $\overline{\Om}$, \eqnref{eqn:resonanceconditionintro} is equivalent to $\big\lVert\nabla \SingleOmega[\varphi^\delta]\big\rVert_{L^2(\Omega)}\rightarrow \infty$ as $\delta\rightarrow 0$.
From \eqnref{eqn:varphideltaint} and \eqnref{eqn:lambda}, we have
\begin{equation}\label{eqn:varphideltanorm}
  \lVert\varphi^\delta\rVert^2_{-1/2}= \int_{-\frac 1 2}^{\frac 1 2}\frac{1}{|\lambda_\delta - t|^2}\,d\big\langle\partial_\nu F,\, \mathbb{E}(t)\partial_\nu F\big\rangle_{-1/2}
\end{equation}
with
\begin{align}\label{eqn:lambdaminust}
	\lambda_\delta - t 
  % =\frac{\epsilon_c+1-2(\epsilon_c-1)t+i\delta(1-2t)}{2(\epsilon_c-1)+2i\delta}
	=\frac{(\lambda_0-t)+\frac{1-2t}{2(\epsilon_c-1)}i\delta}{1+\frac{1}{\epsilon_c-1}i\delta}
	\quad\mbox{with } \lambda_0 = \frac{\epsilon_c+1}{2(\epsilon_c-1)}.
\end{align}
Note that $\lambda$ converges to $\lambda_0$ as $\delta\rightarrow 0$. 
We see from  Lemma \ref{lemma:gradientvalue} and Theorem \ref{thm:spectrumtheoremKstarOmega} that
\begin{align*}
	 \big\lVert\nabla \SingleOmega[\varphi^\delta]\big\rVert_{L^2(\Omega)}^2
	 &=\int_{-\frac 1 2}^{\frac 1 2}\Big(\frac{1}{2}-t\Big)d\big\langle \varphi^\delta,\, \mathbb{E}(t)\,\varphi^\delta\big\rangle_{-1/2}\\
	 &=\int_{-\frac 1 2}^{\frac 1 2}\Big(\frac{1}{2}-t\Big)\frac{1}{|\lambda_\delta - t|^2}\,d\big\langle\partial_\nu F,\, \mathbb{E}(t)\partial_\nu F\big\rangle_{-1/2}.
\end{align*}
For any $\varphi\in K^{-1/2}_0$, we have $\frac{d}{dt}\big\langle \varphi,\, \mathbb{E}(t)\varphi\big\rangle_{-1/2}\geq 0$ from Lemma \ref{lemma:measuredensity}.
Then, \eqnref{eqn:varphideltanorm} leads us to 
\beq\label{Scal:decay:all}
\big\lVert\nabla \SingleOmega[\varphi^\delta]\big\rVert_{L^2(\Omega)}\leq \lVert\varphi^\delta\rVert_{-1/2}.
\eeq

We now assume that $\lambda_0<\frac{1}{2}$. Then, we can take $t_1,t_2$ independent of $\delta$ such that $\lambda_0<t_1<t_2<\frac{1}{2}$.
Then, we have
\begin{align*}
 \big\lVert\nabla \SingleOmega[\varphi^\delta]\big\rVert_{L^2(\Omega)}^2
&=\bigg(\int_{-\frac{1}{2}}^{t_2}+\int_{t_2}^{\frac{1}{2}}\bigg)\Big(\frac{1}{2}-t\Big)\frac{1}{|\lambda_\delta - t|^2}\,d\big\langle\partial_\nu F,\, \mathbb{E}(t)\partial_\nu F\big\rangle_{-1/2}\\
 &\geq \Big(\frac{1}{2} -t_2\Big) \int_{-\frac 1 2}^{t_2}\frac{1}{|\lambda_\delta - t|^2}\,d\big\langle\partial_\nu F,\, \mathbb{E}(t)\partial_\nu F\big\rangle_{-1/2}\\
 &=\Big(\frac{1}{2} -t_2\Big)\bigg( \lVert\varphi^\delta\rVert_{-1/2}^2 - \int_{t_2}^{\frac 1 2} 
 \frac{1}{|\lambda_\delta - t|^2}\,d\big\langle\partial_\nu F,\, \mathbb{E}(t)\partial_\nu F\big\rangle_{-1/2}\bigg)\\
 &\geq \Big(\frac{1}{2} -t_2\Big)\Big(\lVert\varphi^\delta\rVert_{-1/2}^2 - \frac{1}{|\lambda_\delta-t_2|^2}\lVert\p_\nu F\rVert_{-1/2}^2 \Big).
 \end{align*}
Hence, we have
\beq\label{S:varphi:upper}
C_1\lVert\varphi^\delta\rVert_{-1/2} -C_2\leq\big\lVert\nabla \SingleOmega[\varphi^\delta]\big\rVert_{L^2(\Omega)}\leq \lVert\varphi^\delta\rVert_{-1/2}\quad\mbox{for }\lambda_0<\frac{1}{2},
\eeq
where $C_1,C_2$ are some positive constants independent of $\delta$. 
Hence, for $\lambda_0<\frac{1}{2}$, the resonance condition 
$\big\lVert\nabla \SingleOmega[\varphi^\delta]\big\rVert_{L^2(\Omega)}\rightarrow \infty$ as $\delta\rightarrow 0$
is equivalent to that
$\lVert\varphi^\delta\rVert_{-1/2}\rightarrow \infty$ as $\delta \rightarrow 0$. 
In the following, we characterize the resonance depending on $\lambda_0$.

Lemma \ref{lemma:measuredensity} leads us to
\begin{equation}\label{def:functionQ}
d\langle\partial_\nu F,\, \mathbb{E}(t)\partial_\nu F\rangle_{-1/2}=Q(t)\, dt
\end{equation}
with
\begin{align}\label{eqn:functionQ} 
  Q(t)=\begin{dcases}
  \frac{1+2t}{2|t\,{\ln(2|t|)}|}[\,\big| \widehat{f}_1(k)\big|^2+\big|\widehat{f}_1(-k)\big|^2\,]\quad&\text{for }t\in[-1/2,0),\\[1pt]
  \frac{1+2t}{2|t\, {\ln(2|t|)}|}[\,\big|\widehat{f}_2(k)\big|^2+\big|\widehat f_2(-k)\,\big|^2]\quad&\text{for }t\in(0,1/2],
  \end{dcases}
\end{align}
where $k = -\frac{1}{q}\ln(2|t|)$ and $
[\widehat{f}_1, \widehat{f}_2]^T=PU(\partial_\nu F)$.
From \eqnref{Q:L1}, $Q$ belongs to $ L^1(\RR)$. The two functions $\widehat f_1,\widehat f_2$ are bounded, continuous, and vanishing at infinity since $\partial_\nu F$ is smooth on $\p B_R $ and $\p B_r$. Since $\partial_\nu F$ has a zero mean on each $\partial B_r$ and $\partial B_R$, the two components of $U(\partial_\nu F)$ are both zero at $k=0$. Hence, $Q(t)$ is continuous for $t\in(-\frac 1 2,0)\cup (0,\frac 1 2)$. For $t=\pm\frac{1}{2}$, we have 
\begin{align}\label{eqn:Q:limit}
  \lim_{t\rightarrow(-\frac{1}{2})^+}Q(t)=2\big|\widehat f_1(0)\big|^2=0,\quad  \lim_{t\rightarrow\frac{1}{2}^-}Q(t)=2\big|\widehat f_2(0)\big|^2  =0.
\end{align}

In view of \eqnref{S:varphi:upper}, we can determine the order of resonance from the following propositions:
\begin{prop}\label{prop:blowupinterval}
For $\lambda_0\in[-\frac{1}{2},\,0)\cup (0, \frac{1}{2}),$ let $\varphi^\delta$ be the solution to \eqnref{eqn:intgralequationforplasmonicresonance} and $
[\widehat{f}_1, \widehat{f}_2]^T=PU(\partial_\nu F)$.  Then, it holds that
\begin{equation*}
  \lim_{\delta\rightarrow 0} \delta\lVert\varphi^\delta\rVert^2_{-1/2}
  = \frac{|\epsilon_c-1|}{\frac 1 2-\lambda_0}
  \begin{cases}
  Q(\lambda_0)&\text{if }\lambda_0\in(-\frac 1 2,\, 0)\cup(0,\, \frac 1 2),\\[1mm]
 0&\text{if }\lambda_0=-\frac 1 2,
  \end{cases}
\end{equation*}  
where $Q$ is given by \eqnref{eqn:functionQ}.
\end{prop}
\begin{proof}
Let $\lambda_0\in[-\frac{1}{2},\,0)\cup (0, \frac{1}{2})$. Then, we can assume that $\epsilon_c$ is bounded. 
From \eqnref{eqn:varphideltanorm} and \eqnref{eqn:lambdaminust}, we have
\begin{align*}
\frac{1}{|\lambda_\delta - t |^2}
	&=\frac{1+(\frac{1}{\epsilon_c-1}\delta)^2}{(\lambda_0-t)^2+(\frac{1}{2}-t)^2(\frac{1}{\epsilon_c-1}\delta)^2}\nonumber\\\notag
%	&=\frac{1+{\tilde{\delta}}^2}{(1+{\tilde{\delta}}^2)(t-\frac{2\lambda_0+{\tilde{\delta}}^2}{2(1+{\tilde{\delta}}^2)})^2+\frac{{\tilde{\delta}}^2(1-2\lambda_0)^2}{4(1+{\tilde{\delta}}^2)}}
&=\frac{1+{\tilde{\delta}}^2}{(\lambda_0-t)^2+(\frac{1}{2}-t)^2{\tilde{\delta}}^2}
	=\frac{1}{\Big(t-\frac{2\lambda_0+{\tilde{\delta}}^2}{2(1+{\tilde{\delta}}^2)}\Big)^2+\frac{{\tilde{\delta}}^2(1-2\lambda_0)^2}{4(1+{\tilde{\delta}}^2)^2}}
,\quad \tilde{\delta}=\Big|\frac{\delta}{\epsilon_c-1}\Big|.
\end{align*}
We then have from \eqnref{eqn:varphideltanorm} and \eqnref{def:functionQ} that
\begin{align}\label{eqn:infactPoissonkernel}
	\lVert\varphi^\delta\rVert^2_{-1/2}
%	&= \int_{-\frac 1 2}^{\frac 1 2}\frac{1}{|\lambda_\delta - t|^2}\, d\big\langle\partial_\nu F,\, \mathbb{E}(t)\partial_\nu F\big\rangle_{-1/2}\nonumber\\
	&= \frac{\pi\big(1+{\tilde{\delta}}^2\big)}{{\tilde{\delta}}\big(\frac{1}{2}-\lambda_0\big)}\int_{-\infty}^{\infty}\frac{1}{\pi}\frac{\frac{{\tilde{\delta}}(1-2\lambda_0)}{2(1+{\tilde{\delta}}^2)}}{\Big(t-\frac{2\lambda_0+{\tilde{\delta}}^2}{2(1+{\tilde{\delta}}^2)}\Big)^2+\frac{{\tilde{\delta}}^2(1-2\lambda_0)^2}{4(1+{\tilde{\delta}}^2)^2}}\, Q(t)\,\chi_{(-\frac 1 2,\, \frac 1 2)}\, dt.
\end{align}

Now, we simplify \eqnref{eqn:infactPoissonkernel} using the property for the Poisson kernel of the upper half plane:
%The expression \eqnref{eqn:infactPoissonkernel} corresponds to the Poisson kernel for the upper half plane. 
for $g\in L^1(\mathbb{R})$ and $x\in\RR$,
\begin{align}\notag
  \lim_{y\rightarrow 0^+}\frac{1}{\pi}\int_{-\infty}^\infty\frac{y}{(x-t)^2+y^2} g(t)dt
  =\frac{g(x^-)+g(x^+)}{2}
\end{align}
if $g(x^-)$ and $g(x^+)$ exist. Here, $g(x^-)$ and $g(x^+)$ indicate the limit from the left and the right, respectively. From \eqnref{eqn:Q:limit} and the continuity of $Q(t)$ on $t\in(-\frac 1 2,0)\cup (0,\frac 1 2)$, we complete the proof of the lemma. 
\end{proof}

\begin{prop}
Suppose $\lambda_0=0$ (or, equivalently, $\epsilon_c=-1$). Let $\varphi^\delta$ be the solution to \eqnref{eqn:intgralequationforplasmonicresonance}.
Then, it holds that
\begin{equation*}
  \lim_{\delta\rightarrow 0} \delta^2\lVert\varphi^\delta\rVert^2_{-1/2}=0.
\end{equation*}
\end{prop}
\begin{proof}

One can derive that $\lim_{\delta\rightarrow 0} 
 \delta^2\lVert\varphi^\delta\rVert^2_{-1/2}=0$ by following the proof of Proposition 4 in \cite{{Kang:2017:SRN}} with the fact that $Q\in L^1(\RR)$.  
\end{proof}

%%%%%%%%%%%%%%%%%%%%%%%%%
%%%%%%%%%%%%%%%%%%%%%%%%%

\section{Analysis of the NP operator on touching disks}
In this section, we provide the spectral resolution of the NP operator on the touching disks $\Omega=B_R\cup B_{-r}$, following the notation in \eqnref{def:h_a}. The two disks $B_R$ and $B_{-r}$ are tangent to each other at the origin (see the left figure in Figure \ref{fig:exttouchingdomains}). We follow the derivations for the result on a crescent-shaped domain in the previous sections. We omit most proofs since the analysis is almost the same as for the crescent-shaped domain case.

\begin{figure}[b!]
\subfloat[$z$-plane, $z=z_1+iz_2$]{\includegraphics[height=4.5cm, width=5.72cm, trim={3.6cm 2cm 2cm 2cm}, clip]{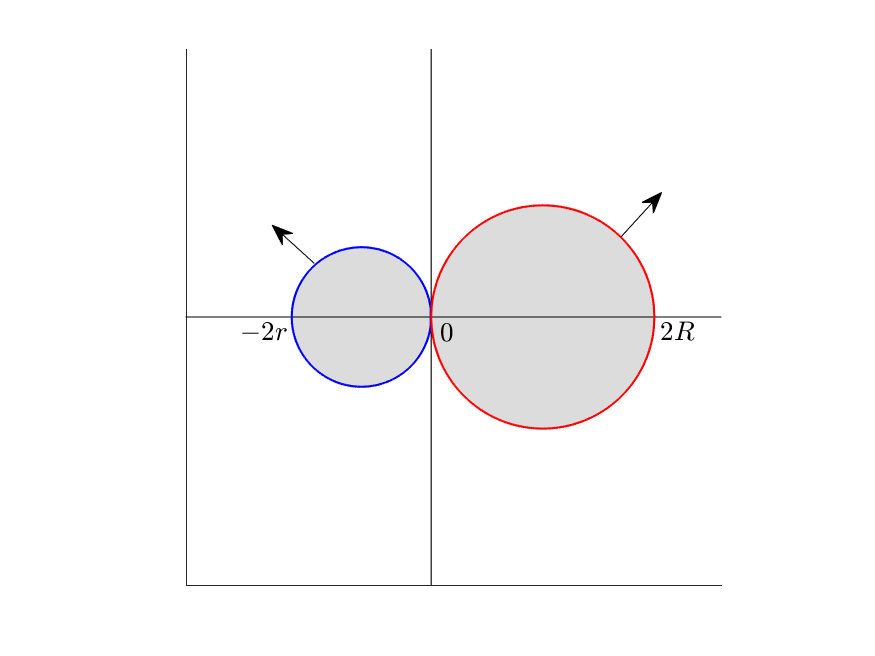}}
\subfloat[$w$-plane, $w=x+iy$]{\includegraphics[height=4.5cm,width=5.72cm, trim={3.6cm 2cm 2cm 2cm}, clip]{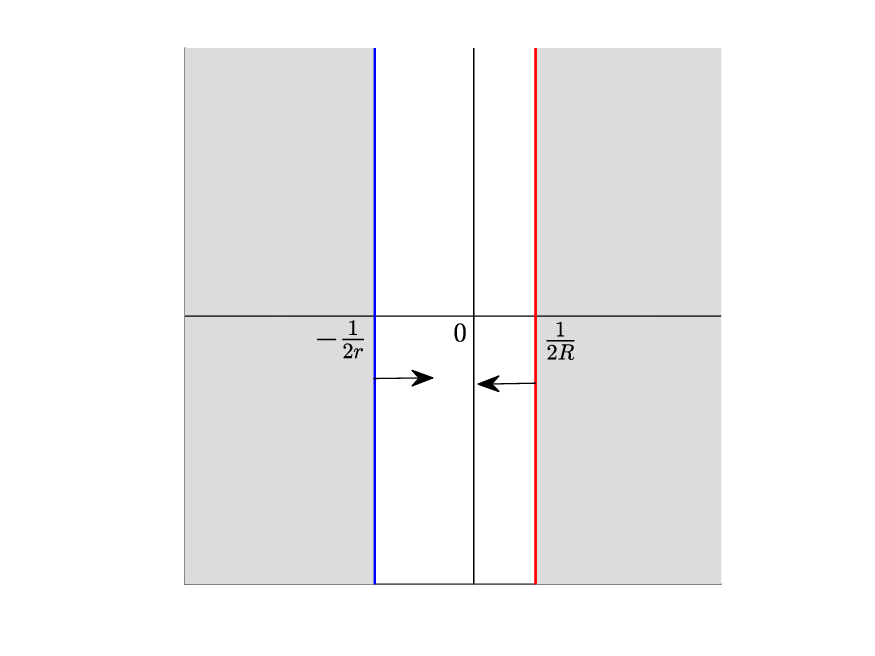}}
\caption{Touching disks $\Omega=B_R\cup B_{-r}$ (gray region in the left figure) and $S=\Psi(\Omega)$ (gray region in the right figure). Arrows indicate the outward normal vectors to $\partial\Omega$ or to $\partial S$.}
\label{fig:exttouchingdomains}
\end{figure}

The touching disks $\Omega$ is mapped onto the region
\begin{equation*}
  S := \Psi(\Omega) = \Big(-\infty,-\frac{1}{2r}\Big)\times
  (-\infty,\infty)\cup \Big(\frac{1}{2R},\infty\Big)\times
  (-\infty,\infty)
\end{equation*}
via the M\"{o}bius mapping $\Psi$ defined in \eqnref{eqn:Mobius}, and vice versa. The outward normal vector convention is described in Figure \ref{fig:exttouchingdomains}. For a function $v$, the normal derivative at $z\in\partial\Omega$ with $\Psi(z)=x+iy$ is
\begin{align*}
  \frac{\partial u}{\partial\nu}&=-\frac{1}{h_R(y)}\frac{\partial(u\circ\Psi)}{\partial x}\quad\text{for } z\in\partial B_R\setminus\{0\},\\
  \frac{\partial u}{\partial\nu}&=\frac{1}{h_{-r}(y)}\frac{\partial(u\circ\Psi)}{\partial x}\quad\text{for } z\in\partial B_{-r}\setminus\{0\},
\end{align*}
where $h_R(y)$ and $h_{-r}(y)$ are defined as in \eqnref{def:h_a}.
Similar to the crescent-shaped domain case, we denote by $\tq$ the distance between the two boundary lines of $S$, i.e., 
\beq\label{def:q:touching}
\tq:=\frac{1}{2R}+\frac{1}{2r}.
\eeq

\subsection{Generalization of the layer potential operators on the crescent-shaped domain}

A density function $\varphi\in\LtwobdOmega$ can be decomposed as
\begin{align}
	\varphi &= \varphi\,\chi_{\p B_R}+\varphi\,\chi_{\partial B_{-r}} \notag \\ \label{def:varphiR_mr}
	&=:\varphi_R +\varphi_{-r}.
\end{align}
We identify $\varphi_R$, $\varphi_{-r}$ with the functions on $\RR$ given by
\beq\label{varphi:tilde2}
\begin{aligned}
	\widetilde{\varphi}_R(y)&=(\varphi_R\circ\Psi^{-1})\Big(\frac{1}{2R}+iy\Big),\\
	\widetilde{\varphi}_{-r}(y)&=(\varphi_{-r}\circ\Psi^{-1})\Big(-\frac{1}{2r}+iy\Big) \quad\mbox{for }y\in\RR.
\end{aligned}
\eeq

We now define the single-layer potential and the NP operator on touching disks. One can easily find that the jump relations in Lemma \ref{SK:relations} holds for touching disks.
\begin{definition}\label{SK:touching}
	For $\varphi=\varphi_R+\varphi_{-r}\in L^2(\p\Om)$, we define 
	\begin{align}\label{def:SignleOmegadef2}
		\SingleOmega[\varphi]
		:=\SingleBR[\varphi_R]
		+ \SingleBmr[\varphi_{-r}]\quad\mbox{on }\RR^2
	\end{align}
	and
	\begin{align*}
		\KstarOmega[\varphi]
		:=&\Big(\KstarBR[\varphi_R] + \frac{\partial}{\partial\nu}\SingleBmr[\varphi_{-r}]\Big|_{\partial B_R}\Big)\chipBR\\
		&+\Big(\frac{\partial}{\partial\nu}\SingleBR[\varphi_R]\Big|_{\partial B_{-r}} + \KstarBmr[\varphi_{-r}]\Big)\chi_{\partial B_{-r}} \quad\text{on } \partial\Omega.
	\end{align*}
\end{definition}

We set $z=\frac{1}{x+iy}$ for $z\neq 0$. Let $z_t=\Psi^{-1}({-\frac{1}{2r}+it})$ on $\p B_{-r}$ and $z_t=\Psi^{-1}(\frac{1}{2R}+it)$ on $\p B_R$; then the single-layer potential \eqnref{def:SignleOmegadef2} satisfies
\beq \label{Single:touching}
\begin{aligned}
  \SingleOmega[\varphi](z) 
  &=\frac{1}{2\pi}\int_{\partial B_{-r}}\ln|z-z_t|\varphi_{-r}(z_t)\, d\sigma (z_t)+\frac{1}{2\pi}\int_{\partial B_{R}}\ln|z-z_t|\varphi_R(z_t)\, d\sigma (z_t)\\
  &=\frac{1}{4\pi}\int_{-\infty}^\infty 
  \bigg(\ln\Big[\Big(x+\frac{1}{2r}\Big)^2+(y-t)^2\Big]-\ln\Big[\Big(\frac{1}{2r}\Big)^2+t^2\Big]\bigg)\tvarphi_{-r}(t)h_{-r}(t)\,dt\\
  &\hskip .2cm +\frac{1}{4\pi}\int_{-\infty}^\infty 
  \bigg(\ln\Big[\Big(x-\frac{1}{2R}\Big)^2+(y-t)^2\Big]-\ln\Big[\Big(\frac{1}{2R}\Big)^2+t^2\Big]\bigg)\tvarphi_R(t)h_R(t)\,dt\\
 &\hskip .2cm -\frac{1}{4\pi}\ln(x^2+y^2)\int_{\partial\Omega}\varphi(z)\, d\sigma(z).
\end{aligned}
\eeq
Thus, we have
\begin{align*}
  \frac{\partial}{\partial\nu}\SingleBmr[\varphi_{-r}]\Big|_{\partial B_R}(z)
  &=-\frac{1}{2\pi h_R(y)}
      \int_{-\infty}^\infty \frac{\tq}{\tq^2+(y-t)^2}\tvarphi_{-r}(t)h_{-r}(t)\, dt+\frac{1}{4\pi R}\int_{-\infty}^\infty\tvarphi_{-r}(t)h_{-r}(t)\, dt,\\
  \frac{\partial}{\partial\nu}\SingleBR[\varphi_R]\Big|_{\partial B_r}(z)
    &=-\frac{1}{2\pi h_{-r}(y)}
      \int_{-\infty}^\infty \frac{\tq}{\tq^2+(y-t)^2} \tvarphi_R(t)h_R(t)\, dt+\frac{1}{4\pi r}\int_{-\infty}^\infty\tvarphi_R(t)h_R(t)\,dt
\end{align*}
with $\tq$ given by \eqnref{def:q:touching}.
The remaining terms of $\KstarOmega$ are
\begin{align*}
  \KstarBR[\varphi_R]&=\frac{1}{4\pi R}\int_{\partial B_R}\varphi_R\, d\sigma\quad\mbox{on }\partial B_R,\\
  \KstarBmr[\varphi_{-r}]&=\frac{1}{4\pi r}\int_{\partial B_{-r}}\varphi_{-r}\, d\sigma\quad\mbox{on }\partial B_{-r}.
\end{align*}
As a result, we obtain that
\begin{equation}\label{K:touching}
	\KstarOmega[\varphi](z)
	 =
	\begin{dcases}
		-\frac{1}{2\pi h_R(y)}
		\int_{-\infty}^\infty \frac{\tq}{\tq^2+(y-t)^2} \tvarphi_{-r}(t) h_{-r}(t) dt +\frac{1}{4\pi R}\int_{\partial\Omega}\varphi\, d\sigma
		\quad&\mbox{for }x=\frac{1}{2R},\\
		- \frac{1}{2\pi h_{-r}(y)}
		\int_{-\infty}^\infty \frac{\tq}{\tq^2+(y-t)^2}\tvarphi_R(t)h_R(t) dt+\frac{1}{4\pi r}\int_{\partial\Omega}\varphi \, d\sigma \quad&\mbox{for }x=-\frac{1}{2r}.
	\end{dcases}
\end{equation}

Similar to the crescent-shaped domain case, we define
\begin{align}
  U[\varphi]
&=
  \begin{bmatrix}
    \mathcal{F}[h_R \widetilde{\varphi}_R]\\
    \mathcal{F}[h_{-r} \widetilde{\varphi}_{-r}]
   \end{bmatrix},
   \end{align}
   where $\mathcal{F}$ is the Fourier transform (see Subsection \ref{subsec:layer:Fourier} for the definition). 
By applying the Fourier transform to \eqnref{Single:touching} and \eqnref{K:touching}, we express the layer potential operators as follows. 
\begin{lemma}\label{lemma:SLFourierDisks}
 Let $\varphi = \varphi_R+\varphi_{-r}\in \LtwozerobdOmega$. For $z=\Psi^{-1}(x+iy)\in \partial\Omega$, we have
\beq\notag
\begin{aligned}
	\SingleOmega[\varphi](z)
	= -\frac{1}{\sqrt{2\pi}}\int_{-\infty}^\infty\frac{1}{2|k|}
	\Big(e^{-|x-\frac{1}{2R}||k|}\mathcal{F}[h_R\widetilde{\varphi}_R](k)
	+e^{-|x+\frac{1}{2r}||k|}\mathcal{F}[h_{-r}\widetilde{\varphi}_{-r}](k) \Big) e^{iky}\, dk+C,  \end{aligned}
\eeq
where $C$ is the constant given by 
\begin{equation*}
  C=\frac{1}{\sqrt{2\pi}}\int_{-\infty}^\infty\frac{1}{2|k|}
  \Big(e^{-\frac{1}{2R}|k|}\mathcal{F}[h_R\widetilde{\varphi}_R](k)
  +e^{-\frac{1}{2r}|k|}\mathcal{F}[h_{-r}\widetilde{\varphi}_{-r}](k) \Big) dk.
\end{equation*}
\end{lemma}
\begin{lemma}\label{lemma:KstarFourierDisks}
For $\varphi\in \LtwozerobdOmega$, it holds that
\begin{equation*}
  U[\KstarOmega[\varphi]]
   = P^{-1}\Big(\frac{1}{2}e^{-|k|\tq}
   \begin{bmatrix}
    1&0\\0&-1
   \end{bmatrix}
   P\Big)U[\varphi]
\end{equation*}
with $\tq$ given by \eqnref{def:q:touching}.
\end{lemma}

\subsection{Spectral resolution of the NP operator on touching disks}
We define two matrix-valued functions in Lemma \ref{lemma:SLFourierDisks} and Lemma \ref{lemma:KstarFourierDisks} as 
\begin{equation}\label{def2:matrix:S:K}
\begin{aligned}
  \widetilde{\mathbb{S}}&= -\frac{1}{2|k|}\begin{bmatrix}
     1-e^{-\tq|k|}&0\\
     0 &1+e^{-\tq|k|}
   \end{bmatrix},\quad k\in\mathbb{R}\setminus\{0\},\\[2mm]
     \widetilde{\mathbb{K}}&=\frac{1}{2}e^{-\tq|k|}
  \begin{bmatrix}
    1&0\\0&-1
   \end{bmatrix},\quad k\in\mathbb{R}.  
   \end{aligned}
\end{equation}
Note that the definition of the matrix $\widetilde{\mathbb{S}}$ is identical to $\mathbb{S}$ of the crescent-shaped domain case in \eqnref{def:matrix:S:K}, except that $q$ is now replaced by $\tq$. Note also that the signs of the diagonal entries of $\widetilde{\mathbb{K}}$ are changed from the diagonal entries of $\mathbb{K}$ of the crescent-shaped domain case (see \eqnref{def:matrix:S:K}).

\begin{definition}\label{def:K_02}
We define $K^{-1/2}_0$ as the same as Definition \ref{def:K_0} with $\mathbb{S}$ replaced by $\widetilde{\mathbb{S}}$. In other words,\beq\notag
	K^{-1/2}_0:=\bigg\{\varphi=U^{-1}P\hvarphi\ \,\Big\vert\, \hvarphi=\begin{bmatrix}\widehat{\varphi}_1\\\widehat{\varphi}_2
	\end{bmatrix}\mbox{ satisfying } \int_{-\infty}^\infty
	\hvarphi^T
	(-\widetilde{\mathbb{S}})~
	\overline{\hvarphi} \, dk<\infty\bigg\} ,
	\eeq  where $\hvarphi_1$ and $\hvarphi_2$ are measurable functions on $\RR$, and $P$ is given by \eqnref{P:def}.
	
	Also, we define the inner product $\la \cdot,\cdot\ra_{-1/2}$ and the norm $\lVert\cdot\rVert_{-1/2}$ on $K^{-1/2}_0$ as the same as in Section \ref{subsection:newspaceNPoperator} with $\mathbb{S}$ replaced by $\widetilde{\mathbb{S}}$.
	\end{definition}

We now naturally extend the single-layer potential and the NP operator on $\LtwozerobdOmega$ to $K^{-1/2}_0$ by generalizing the formulas in Lemma \ref{lemma:SLFourierDisks} and Lemma \ref{lemma:KstarFourierDisks}, respectively.

\begin{definition}
		Let $\varphi\in  K^{-1/2}_0$ be given by $\varphi=U^{-1}P\hvarphi$.
		\begin{enumerate}[(a), left=0.5em]
			\item  
			We define the NP operator $\KstarOmega:K^{-1/2}_0\rightarrow K^{-1/2}_0$ 
			by
			\begin{equation}
				\KstarOmega[\varphi]:=U^{-1}P\widetilde{\mathbb{K}}\hvarphi.
			\end{equation}
			
			\item
			We define the single-layer potential of $\varphi$: for $z=\Psi^{-1}(x+iy)\in\CC$,
			\beq 
			\begin{aligned}
				\SingleOmega[\varphi](z):= &-\frac{1}{\sqrt{2\pi}}\int_{-\infty}^\infty\frac{1}{2|k|}
				\Big(e^{-|x-\frac{1}{2R}||k|}\,\hvarphi_R(k)
				+e^{-|x+\frac{1}{2r}||k|}\,\hvarphi_r(k)\Big) e^{iky}\, dk \\
				&+\frac{1}{\sqrt{2\pi}}\int_{-\infty}^\infty\frac{1}{2|k|}
				\Big(e^{-\frac{|k|}{2R}}\,\hvarphi_R(k)+e^{-\frac{|k|}{2r}}\,\hvarphi_r(k)\Big)\, dk,
			\end{aligned}
			\eeq
			where $\hvarphi_R$ and $\hvarphi_r$ is given by 
			\beq
			\begin{bmatrix}
				\hvarphi_R\\ 
				\hvarphi_r
			\end{bmatrix}
			=P
			\begin{bmatrix}
				\hvarphi_1\\
				\hvarphi_2
			\end{bmatrix}
			=\frac{1}{\sqrt{2}}
			\begin{bmatrix}
				-\hvarphi_1+\hvarphi_2\\
				\hvarphi_1+\hvarphi_2
			\end{bmatrix}.
			\eeq
		\end{enumerate}
\end{definition}

From arguments similar to those for the crescent-shaped domain case, it can be shown that $\KstarOmega$ is a bounded linear operator on $K^{-1/2}_0$ whose operator norm is bounded by $\frac{1}{2}$. In addition, $\KstarOmega$ is self-adjoint on $K^{-1/2}_0$, and the spectrum of $\KstarOmega$ on $K^{-1/2}_0$ lies in the interval $[-1/2,1/2]$. In the remainder of this subsection, we derive the spectral resolution of the operator. 

To derive the spectral resolution of the NP operator on $K^{-1/2}_0$, we define a pair of orthogonal projection operators $\mathcal{P}_1(s)$ and $\mathcal{P}_2(s)$ on $K^{-1/2}_0$ for each $s\in\RR\cup\{\infty\}$.
Let $\varphi\in K^{-1/2}_0$ be given by $\varphi=U^{-1}P\hvarphi$ and $\hvarphi=[\hvarphi_1,\hvarphi_2]^T$. We define 
\begin{align*}
	\mathcal{P}_1(s)\varphi&= U^{-1}P\begin{bmatrix}
		\chi_{(-\infty,\, s\,]}\, \widehat{\varphi}_1\\[2mm]
		0
	\end{bmatrix},\quad
	\mathcal{P}_2(s)\varphi = U^{-1}P\begin{bmatrix}
		0\\
		\chi_{(-\infty,\, s\,]}\, \widehat{\varphi}_2
	\end{bmatrix}
\end{align*}
for $s\in\RR$ and 
\begin{equation*}
  \mathcal{P}_1(\infty)\varphi=U^{-1}P\begin{bmatrix}
  	\widehat{\varphi}_1\\[2mm]
  	0
  \end{bmatrix}
  ,\quad
  \mathcal{P}_2(\infty)\varphi = U^{-1}P\begin{bmatrix}
  	0\\
  	\widehat{\varphi}_2
  \end{bmatrix}.
\end{equation*}
Note that $\mathbb{I}=\mathcal{P}_2(\infty)+\mathcal{P}_1(\infty)$, where $\mathbb{I}$ is the identity operator on $K^{-1/2}_0$. Now, we define a family of projection operators $\mathbb{E}(t)$ on $K^{-1/2}_0$, $t\in [-1/2,1/2]$, as in \eqnref{def:resolutionoftheidentity}. 
We then define
\beq\label{def:tildeE}
\widetilde{\mathbb{E}}(t):=\mathbb{I}-\mathbb{E}(-t).
\eeq
In other words, 
%\begin{equation}\label{def:resolutionoftheidentity2}
%	\mathbb{E}(t)
%	:=\begin{cases}
%		 \mathcal{P}_1\Big(-\frac{\ln(-2t)}{q_1}\Big)-\mathcal{P}_1\Big(\frac{\ln(-2t)}{q_1} \Big), \quad&t\in[-1/2,0),\\[2mm]
%		 \mathcal{P}_1(\infty),\quad&t=0,\\[2mm]
%		  -\mathcal{P}_2\Big(-\frac{\ln(2t)}{q_1} \Big)+\mathcal{P}_2\Big(\frac{\ln(2t)}{q_1}\Big)+\mathbb{I}, \quad&t\in(0, 1/2].
%	\end{cases}
%\end{equation}
\begin{equation}\label{def:resolutionoftheidentity2}
	\widetilde{\mathbb{E}}(t)
	:=\begin{cases}
		  \mathcal{P}_2\Big(-\frac{\ln(-2t)}{\tq} \Big)-\mathcal{P}_2\Big(\frac{\ln(-2t)}{\tq}\Big), \quad&t\in[-1/2,0),\\
		 \mathcal{P}_2(\infty),\quad&t=0,\\[2mm]
		 -\mathcal{P}_1\Big(-\frac{\ln(2t)}{\tq}\Big)+\mathcal{P}_1\Big(\frac{\ln(2t)}{\tq} \Big)+\mathbb{I}, \quad&t\in(0,1/2].
	\end{cases}
\end{equation}
Then, the family of operators $\{\widetilde{\mathbb{E}}(t)\}_{t\in[-\frac 1 2, \frac 1 2]}$ is a resolution of identity on $K^{-1/2}_0$ and satisfies
 \beq\label{prop:innerproductrep2}
 \lim_{t\rightarrow s}\widetilde{\mathbb{E}}(t)=\widetilde{\mathbb{E}}(s)
 \eeq
 in the sense of strong convergence.

Following the derivations of \eqnref{psi:E:phi} and \eqnref{eqn:spectrumtheoremKstarOmega} for the case of the crescent-shaped domain, it can be shown that, for all $\psi,\varphi\in K^{-1/2}_0$, 
\begin{gather}\notag
		\langle \psi,\varphi\rangle_{-1/2}=\int_{-\frac{1}{2}}^{\frac 1 2}d\big\langle \psi,\, \mathbb{E}(t)\varphi\big\rangle_{-1/2},\\
\big\langle \psi, \KstarOmega[\varphi] \big\rangle_{-1/2}= -\int_{-\frac 1 2}^{\frac 1 2}\,t\, d\,\big\langle \psi, \mathbb{E}(t)\varphi \big\rangle_{-1/2}.\label{psi:K:phi:2}
\end{gather}
The right-hand side of \eqnref{psi:K:phi:2} has different sign of from that of \eqnref{eqn:spectrumtheoremKstarOmega} because the signs of the diagonal entries of $\widetilde{\mathbb{K}}$ are changed from the diagonal entries of $\mathbb{K}$ in \eqnref{def:matrix:S:K}. 
In view of \eqnref{def:tildeE}, we obtain the following. 
\begin{lemma}\label{lemma:resolution2}
	For all $\varphi,\psi\in K^{-1/2}_0$, it holds that
	\begin{equation}\notag
		\langle \psi,\varphi\rangle_{-1/2}=\int_{-\frac{1}{2}}^{\frac 1 2}d\big\langle \psi,\, \widetilde{\mathbb{E}}(t)\varphi\big\rangle_{-1/2}.
	\end{equation}
\end{lemma}

\begin{theorem}\label{thm:spectrumtheoremKstarOmega2}
	Let $\big\{\widetilde{\mathbb{E}}(t)\big\}_{t\in[-1/2,1/2]}$ be the resolution of the identity on $K^{-1/2}_0$ given by \eqnref{def:resolutionoftheidentity2}. Then we have the following spectral resolution of $\KstarOmega$ on $K^{-1/2}_0$:
	\begin{equation*}
		\KstarOmega = \int_{-\frac 1 2}^{\frac 1 2}\,t\, d\,\widetilde{\mathbb{E}}(t).
	\end{equation*}
	In other words, it holds that
	\begin{equation}\label{eqn2:spectrumtheoremKstarOmega}
		\big\langle \psi, \KstarOmega[\varphi] \big\rangle_{-1/2}= \int_{-\frac 1 2}^{\frac 1 2}\,t\, d\,\big\langle \psi, \widetilde{\mathbb{E}}(t)\varphi \big\rangle_{-1/2}\quad\mbox{for all }\psi,\varphi\in K^{-1/2}_0.
	\end{equation}
\end{theorem}

The condition $\lim_{t\rightarrow s^-}\widetilde{\mathbb{E}}(t) \neq \widetilde{\mathbb{E}}(s)$
characterizes the point spectrum of $\KstarOmega$.
From \eqnref{prop:innerproductrep2}, we conclude that $\KstarOmega$ has only a continuous spectrum. 
In view of \eqnref{def:tildeE} and Lemma \ref{lemma:measuredensity}, we have 
\begin{equation*}
  \frac{d}{dt}\big\langle \psi,\, \widetilde{\mathbb{E}}(t)\varphi\big\rangle_{-1/2}
  		=\frac{d}{ds}\big\langle \psi,\, {\mathbb{E}}(s)\varphi\big\rangle_{-1/2}{\bigg|}_{s=-t}
\end{equation*}
and
	\begin{align*}
		\frac{d}{dt}\big\langle \psi,\, \widetilde{\mathbb{E}}(t)\varphi\big\rangle_{-1/2}
		=\begin{dcases}
			  \frac{1-2t}{2|t\ln(2|t|)|}\Big(\,\hpsi_1(k)\, \overline{ \hvarphi_1(k)}+\hpsi_1(-k)\, \overline{\hvarphi_1(-k)}\, \Big) ,\quad&t\in[-1/2,0),\\[1mm]
			  \frac{1-2t}{2|t\ln(2|t|)|}\Big(\,\hpsi_2(k)\, \overline{ \hvarphi_2(k)}+\hpsi_2(-k)\, \overline{\hvarphi_2(-k)}\, \Big), \quad&t\in(0,1/2]
		\end{dcases}
	\end{align*}
	with  $k = -\frac{1}{\tq}\ln(2|t|)$. Similar to \eqnref{eqn2:lemma:measuredensity}, it holds that
\begin{align}\notag
  \int_{-\frac 1 2}^{\frac 1 2}\bigg|\frac{d}{dt}\big\langle \varphi,\, \widetilde{\mathbb{E}}(t)\varphi\big\rangle_{-1/2}\bigg|\, dt 
 &= \frac{1}{2}\int_{-\infty}^\infty\bigg(\frac{1-e^{-\tq|k|}}{|k|}\, |\hvarphi_1(k)|^2+\frac{1+e^{-\tq|k|}}{|k|}\, |\hvarphi_2(k)|^2\bigg)\, dk<\infty.
\end{align}
Finally, we obtain the following theorem.
\begin{theorem}
	The NP operator $\KstarOmega$ on $K^{-1/2}_0$ has only the absolutely continuous spectrum $[-\frac 1 2,\frac 1 2]$.
\end{theorem}

\bibliography{2020_Jung_SAN}

\ifx \bblindex \undefined \def \bblindex #1{} \fi\ifx \bbljournal \undefined
  \def \bbljournal #1{{\em #1}\index{#1@{\em #1}}} \fi\ifx \bblnumber
  \undefined \def \bblnumber #1{{\bf #1}} \fi\ifx \bblvolume \undefined \def
  \bblvolume #1{{\bf #1}} \fi\ifx \noopsort \undefined \def \noopsort #1{} \fi
\begin{thebibliography}{10}

\bibitem{Ammari:2013:STN}
Habib Ammari, Giulio Ciraolo, Hyeonbae Kang, Hyundae Lee, and Graeme~W. Milton.
\newblock Spectral theory of a {N}eumann-{P}oincar\'{e}-type operator and
  analysis of cloaking due to anomalous localized resonance.
\newblock {\em Arch. Ration. Mech. Anal.}, 208(2):667--692, 2013.

\bibitem{Ammari:2007:BIM}
Habib Ammari, Hyeonbae Kang, and Hyundae Lee.
\newblock A boundary integral method for computing elastic moment tensors for
  ellipses and ellipsoids.
\newblock {\em J. Comput. Math.}, 25(1):2--12, 2007.

\bibitem{Ando:2019:SSN}
Kazunori Ando, Yong-Gwan Ji, Hyeonbae Kang, Daisuke Kawagoe, and Yoshihisa
  Miyanishi.
\newblock Spectral structure of the {N}eumann-{P}oincar\'{e} operator on tori.
\newblock {\em Ann. Inst. H. Poincar\'{e} Anal. Non Lin\'{e}aire},
  36(7):1817--1828, 2019.

\bibitem{Ando:2016:APR}
Kazunori Ando and Hyeonbae Kang.
\newblock Analysis of plasmon resonance on smooth domains using spectral
  properties of the {N}eumann-{P}oincar\'{e} operator.
\newblock {\em J. Math. Anal. Appl.}, 435(1):162--178, 2016.

\bibitem{Ando:2016:PRF}
Kazunori Ando, Hyeonbae Kang, and Hongyu Liu.
\newblock Plasmon resonance with finite frequencies: a validation of the
  quasi-static approximation for diametrically small inclusions.
\newblock {\em SIAM J. Appl. Math.}, 76(2):731--749, 2016.

\bibitem{Ando:2018:EDE}
Kazunori Ando, Hyeonbae Kang, and Yoshihisa Miyanishi.
\newblock Exponential decay estimates of the eigenvalues for the
  {N}eumann-{P}oincar\'{e} operator on analytic boundaries in two dimensions.
\newblock {\em J. Integral Equations Appl.}, 30(4):473--489, 2018.

\bibitem{Ando:2020:SAN}
Kazunori Ando, Hyeonbae Kang, Yoshihisa Miyanishi, and Mihai Putinar.
\newblock Spectral analysis of {Neumann-Poincar\'{e}}.
\newblock {\em arXiv:2003.14387}, 2020.

\bibitem{Aubry:2010:BPD}
Alexandre Aubry, Dang~Yuan Lei, Stefan~A Maier, and JB~Pendry.
\newblock Broadband plasmonic device concentrating the energy at the nanoscale:
  {T}he crescent-shaped cylinder.
\newblock {\em Physical Review B}, 82(12):125430, 2010.

\bibitem{Blumenfeld:1914:UPF}
J~Blumenfeld and W~Mayer.
\newblock {\"U}ber {Poincar{\'e}sche} {F}undamentalfunktionen.
\newblock {\em Sitz. Wien. Akad. Wiss., Math.-Nat. Klasse Abt. IIa},
  122:2011--2047, 1914.

\bibitem{Dhia:2020:CMP}
Anne-Sophie Bonnet-Ben~Dhia, Christophe Hazard, and Florian Monteghetti.
\newblock Complex-scaling method for the complex plasmonic resonances of planar
  subwavelength particles with corners.
\newblock {\em J. Comput. Phys.}, 440:Paper No. 110433, 29, 2021.

\bibitem{Bonnetier:2019:PRB}
Eric Bonnetier, Charles Dapogny, Faouzi Triki, and Hai Zhang.
\newblock The plasmonic resonances of a bowtie antenna.
\newblock {\em Anal. Theory Appl.}, 35(1):85--116, 2019.

\bibitem{Bonnetier:2012:PBO}
Eric Bonnetier and Faouzi Triki.
\newblock Pointwise bounds on the gradient and the spectrum of the
  {N}eumann-{P}oincar\'{e} operator: the case of 2 discs.
\newblock In {\em Multi-scale and high-contrast {PDE}: from modelling, to
  mathematical analysis, to inversion}, volume 577 of {\em Contemp. Math.},
  pages 81--91. Amer. Math. Soc., Providence, RI, 2012.

\bibitem{Bonnetier:2013:SPV}
Eric Bonnetier and Faouzi Triki.
\newblock On the spectrum of the {P}oincar\'{e} variational problem for two
  close-to-touching inclusions in 2{D}.
\newblock {\em Arch. Ration. Mech. Anal.}, 209(2):541--567, 2013.

\bibitem{Bonnetier:2019:CES}
Eric Bonnetier and Hai Zhang.
\newblock Characterization of the essential spectrum of the
  {N}eumann-{P}oincar\'{e} operator in 2{D} domains with corner via {W}eyl
  sequences.
\newblock {\em Rev. Mat. Iberoam.}, 35(3):925--948, 2019.

\bibitem{Escauriaza:1992:RTW}
L.~Escauriaza, E.~B. Fabes, and G.~Verchota.
\newblock On a regularity theorem for weak solutions to transmission problems
  with internal {L}ipschitz boundaries.
\newblock {\em Proc. Amer. Math. Soc.}, 115(4):1069--1076, 1992.

\bibitem{Escauriaza:2004:TPS}
Luis Escauriaza and Marius Mitrea.
\newblock Transmission problems and spectral theory for singular integral
  operators on {L}ipschitz domains.
\newblock {\em J. Funct. Anal.}, 216(1):141--171, 2004.

\bibitem{Fabes:1992:SRC}
Eugene Fabes, Mark Sand, and Jin~Keun Seo.
\newblock The spectral radius of the classical layer potentials on convex
  domains.
\newblock In {\em Partial differential equations with minimal smoothness and
  applications ({C}hicago, {IL}, 1990)}, volume~42 of {\em IMA Vol. Math.
  Appl.}, pages 129--137. Springer, New York, 1992.

\bibitem{Feng:2016:SNP}
Tingting Feng and Hyeonbae Kang.
\newblock Spectrum of the {N}eumann-{P}oincar\'{e} operator for ellipsoids and
  tunability.
\newblock {\em Integral Equations Operator Theory}, 84(4):591--599, 2016.

\bibitem{Grieser:2014:PEP}
Daniel Grieser.
\newblock The plasmonic eigenvalue problem.
\newblock {\em Rev. Math. Phys.}, 26(3):1450005, 26, 2014.

\bibitem{Helsing:2017:CSN}
Johan Helsing, Hyeonbae Kang, and Mikyoung Lim.
\newblock Classification of spectra of the {N}eumann--{P}oincar\'{e} operator
  on planar domains with corners by resonance.
\newblock {\em Ann. Inst. H. Poincar\'{e} Anal. Non Lin\'{e}aire},
  34(4):991--1011, 2017.

\bibitem{Jung:2020:DEE}
Younghoon Jung and Mikyoung Lim.
\newblock A decay estimate for the eigenvalues of the {N}eumann-{P}oincar\'{e}
  operator using the {G}runsky coefficients.
\newblock {\em Proc. Amer. Math. Soc.}, 148(2):591--600, 2020.

\bibitem{Kang:2016:SPO}
Hyeonbae Kang, Kyoungsun Kim, Hyundae Lee, Jaemin Shin, and Sanghyeon Yu.
\newblock Spectral properties of the {N}eumann-{P}oincar\'{e} operator and
  uniformity of estimates for the conductivity equation with complex
  coefficients.
\newblock {\em J. Lond. Math. Soc.}, 93(2):519--545, 2016.

\bibitem{Kang:2017:SRN}
Hyeonbae Kang, Mikyoung Lim, and Sanghyeon Yu.
\newblock Spectral resolution of the {N}eumann--{P}oincar\'{e} operator on
  intersecting disks and analysis of plasmon resonance.
\newblock {\em Arch. Ration. Mech. Anal.}, 226(1):83--115, 2017.

\bibitem{Kang:2018:SPS}
Hyeonbae Kang and Mihai Putinar.
\newblock Spectral permanence in a space with two norms.
\newblock {\em Rev. Mat. Iberoam.}, 34(2):621--635, 2018.

\bibitem{Kellogg:1953:FPT}
Oliver~Dimon Kellogg.
\newblock {\em Foundations of potential theory}.
\newblock Dover, New York, 1953.
\newblock Originally published in 1929 by J. Springer.

\bibitem{Khavinson:2007:PVP}
Dmitry Khavinson, Mihai Putinar, and Harold~S. Shapiro.
\newblock {P}oincar\'{e}'s variational problem in potential theory.
\newblock {\em Arch. Ration. Mech. Anal.}, 185(1):143--184, 2007.

\bibitem{Krein:1998:CLO}
M.~G. Krein.
\newblock Compact linear operators on functional spaces with two norms.
\newblock {\em Integr. Equat. Oper. Th.}, 30(2):140--162, 1998.

\bibitem{Li:2020:IME}
Wei Li, Karl-Mikael Perfekt, and Stephen~P. Shipman.
\newblock {Infinitely many embedded eigenvalues for the Neumann--Poincar\'e
  operator in 3D}.
\newblock {\em arXiv:2009.04371}, 2020.

\bibitem{Li:2019:EEN}
Wei Li and Stephen~P. Shipman.
\newblock Embedded eigenvalues for the {N}eumann-{P}oincar\'{e} operator.
\newblock {\em J. Integral Equations Appl.}, 31(4):505--534, 2019.

\bibitem{Lim:2001:SBI}
Mikyoung Lim.
\newblock Symmetry of a boundary integral operator and a characterization of a
  ball.
\newblock {\em Illinois J. Math.}, 45(2):537--543, 2001.

\bibitem{Lim:2015:AOT}
Mikyoung Lim and Sanghyeon Yu.
\newblock Asymptotics of the solution to the conductivity equation in the
  presence of adjacent circular inclusions with finite conductivities.
\newblock {\em J. Math. Anal. Appl.}, 421(1):131--156, 2015.

\bibitem{Mayergoyz:2005:EPR}
Isaak~D Mayergoyz, Donald~R Fredkin, and Zhenyu Zhang.
\newblock Electrostatic (plasmon) resonances in nanoparticles.
\newblock {\em Phys. Rev. B}, 72(15):155412, 2005.

\bibitem{Milton:2006:CEA}
Graeme~W. Milton and Nicolae-Alexandru~P. Nicorovici.
\newblock On the cloaking effects associated with anomalous localized
  resonance.
\newblock {\em Proc. R. Soc. Lond. Ser. A Math. Phys. Eng. Sci.},
  462(2074):3027--3059, 2006.

\bibitem{Miyanishi:2017:EED}
Yoshihisa Miyanishi and Takashi Suzuki.
\newblock Eigenvalues and eigenfunctions of double layer potentials.
\newblock {\em Trans. Amer. Math. Soc.}, 369(11):8037--8059, 2017.

\bibitem{Perfekt:2021:PEP}
Karl-Mikael Perfekt.
\newblock Plasmonic eigenvalue problem for corners: limiting absorption
  principle and absolute continuity in the essential spectrum.
\newblock {\em J. Math. Pures Appl. (9)}, 145:130--162, 2021.

\bibitem{Perfekt:2014:SBN}
Karl-Mikael Perfekt and Mihai Putinar.
\newblock Spectral bounds for the {N}eumann--{P}oincar\'{e} operator on planar
  domains with corners.
\newblock {\em J. Anal. Math.}, 124:39--57, 2014.

\bibitem{Perfekt:2017:ESN}
Karl-Mikael Perfekt and Mihai Putinar.
\newblock The essential spectrum of the {N}eumann--{P}oincar\'{e} operator on a
  domain with corners.
\newblock {\em Arch. Ration. Mech. Anal.}, 223(2):1019--1033, 2017.

\bibitem{Teschl:2009:MMQ}
Gerald Teschl.
\newblock {\em Mathematical methods in quantum mechanics}, volume~99 of {\em
  Graduate Studies in Mathematics}.
\newblock American Mathematical Society, Providence, RI, 2009.
\newblock With applications to Schr\"{o}dinger operators.

\bibitem{Verchota:1984:LPR}
Gregory Verchota.
\newblock Layer potentials and regularity for the {D}irichlet problem for
  {L}aplace's equation in {L}ipschitz domains.
\newblock {\em J. Funct. Anal.}, 59(3):572--611, 1984.

\bibitem{Yosida:1995:FA}
K\^{o}saku Yosida.
\newblock {\em Functional analysis}.
\newblock Classics in Mathematics. Springer-Verlag, Berlin, 1995.
\newblock Reprint of the sixth (1980) edition.

\end{thebibliography}
\bibliographystyle{plain}
% \begin{thebibliography}{10}

% \end{thebibliography}

\end{document}